\documentclass[12pt]{article}
\usepackage{evan}
\usepackage{bbm}
\usepackage{mdframed}

\definecolor{shadecolor}{rgb}{0.9, 0.9, 0.9}

\usepackage{fullpage,amsfonts,amsmath,amsthm,amssymb,mathrsfs,epstopdf}
\usepackage[top=2.54cm, bottom=2.54cm, left=2.54 cm, right=2.54 cm]{geometry}
\usepackage{graphicx}
\usepackage{caption}
\usepackage{subcaption}
\usepackage{titling} 

\usepackage[shortlabels]{enumitem}
\usepackage{thmtools}
\newtheorem{lemma}[theorem]{Lemma}
\newtheorem*{definition*}{Definition}
\newtheorem*{remark*}{Remark}
\newtheorem*{example*}{Example}
\newtheorem*{theorem*}{Theorem}

\def\pr{{\mathbb P}}
\def\ex{{\mathbb E}}

\begin{document}

\predate{}
\postdate{}

\author{Thinula De Silva\\ University of Waterloo\\ t2desilva@uwaterloo.ca,\and Pu Gao\thanks{Research supported by NSERC RGPIN-04173-2019.}\\ University of Waterloo\\ pu.gao@uwaterloo.ca}
\title{Non-uniform Kahn-Kalai, spread, variants, and applications}
\date{} 

\maketitle

\def\G{{\mathcal G}}
\def\F{{\mathcal F}}
\def\P{{\mathcal P}}
\def\bfU{{\bf U}}

\renewcommand{\vec}[1]{\textbf{#1}}
\newcommand{\prob}[1]{\mathbb{P}\left(#1\right)}
\newcommand{\expect}[1]{\mathbb{E}\left(#1\right)}
\newcommand{\var}[1]{\operatorname{Var}\left(#1\right)}
\newcommand{\C}{\mathbb{C}}
\newcommand{\R}{\mathbb{R}}
\newcommand{\N}{\mathbb{N}}
\newcommand{\Z}{\mathbb{Z}}
\newcommand{\gen}[1]{\left\langle #1 \right\rangle}
\renewcommand{\vocab}[1]{{\it #1}}
\newcommand{\notetoself}[1]{\textbf{\color{blue} #1}}

\newcommand{\jc}[1]{{\color{red}[Jane: #1]}}
\newcommand{\tc}[1]{{\color{olive}[Thinula: #1]}}
\newcommand{\tr}[1]{\color{violet} #1}
\newcommand{\remove}[1]{{}}

\newcommand\pdfmath[1]{\texorpdfstring{$#1$}{}}

\newcommand{\nvec}{\vec{n}}

\begin{abstract}
Building on B.~Park and Vondr{\'a}k's recent generalization of the J.Park-Pham Theorem (formerly known as Kahn-Kalai conjecture) to non-uniform probability measures, this paper introduces the notion of ``spread" for the non-uniform setting. This provides a framework to establish 1-statements for subgraph containment in inhomogeneous random graphs with or without a set of forced edges. Using this approach, we derived conditions for the emergence of perfect matchings in the Stochastic Block Model and the Chung-Lu model, and verified that these conditions are in general not tight, but they capture thresholds across a broad range of regimes. Finally, we bridge this non-uniform framework with $\mathcal{G}(n,\vec{d})$, utilizing a coupling argument to demonstrate thresholds for perfect matchings in $\mathcal{G}(n,\vec{d})$ for a broad range of degree sequences $\vec{d}$. 
\end{abstract}

\section{Introduction}

Random graph theory, initiated by Erd\H{o}s and R\'{e}nyi in the late 1950s, has become a central field in modern graph theory and theoretical computer science. One of the most striking phenomena in random graph theory is the abrupt structural transformation known as a phase transition, where a small change in edge probability leads to a sudden change in graph structures or parameters. 
Friedgut and Kalai proved that every monotone graph property exhibits a phase transition~\cite{Friedgut1996EveryMG}. Determining the threshold at which a given monotone graph property undergoes the phase transition is a problem-specific task  and lies at the heart of random graph theory. The famous Kahn-Kalai conjecture, proposed by Kahn and Kalai in 2006, relates the threshold of an increasing family to a natural lower bound called the expectation threshold. The conjecture asserts that this lower bound is, up to a logarithmic factor, close to the actual threshold.

This conjecture is stated in a more general setting.
Let $X$ be a finite set. In the example of random graph $\G(n,p)$, $X=\binom{[n]}{2}$ is the set of edges in a complete graph on $n$ vertices. Let $\mathcal{F}\subseteq 2^X$ be an increasing family, where $\mathcal{F}\neq\emptyset, 2^X$. As an example of increasing graph properties, we can think of ${\mathcal F}$ being ${\mathcal F}_{\texttt{PM}}$, the set of graphs on $[n]$ containing a perfect matching (assuming that $n$ is even). Given $p\in [0,1]$, a probability measure $\mu_p$ on the set of subsets of $X$ is defined by
\[
\mu_p(S) = p^{\abs{S}}(1-p)^{\abs{X\setminus S}} \quad \text{for all $S\subseteq X$}.
\]
It is not hard to show that
\[
\mu_p(\mathcal{F}) := \sum\limits_{S\in\mathcal{F}} \mu_p(S)
\]
 is strictly increasing in $p$ with $\mu_0(\F)=0$ and $\mu_1(\F)=1$. Thus, there is a unique value $p_c(\mathcal{F})$, called the \vocab{threshold} of $\F$, such that $u_p(\mathcal{F}) = 1/2$ when $p=p_c(\mathcal{F})$. As noted in \cite{Bollobs1987ThresholdF},
\begin{center}
    $\lim\limits_{n\rightarrow\infty} \mu_p(\mathcal{F}) = \begin{cases} 0 &\text{ if } p\ll p_c(\mathcal{F})\\ 1 &\text{ if } p\gg p_c(\mathcal{F})\end{cases}$
\end{center}
and thus $p_c(\mathcal{F})$ marks the phase transition where the measure $\F$ undergoes the sudden change from 0 to 1.

Determining $p_c(\mathcal{F})$ involves identifying $p_c$, usually a function of $N=|X|$, and proving that 
\[
{\text{{\em 0-statement:}}}\quad \mu_p(\F)=o(1), \quad\text{if $p\ll p_c$};
\]
and
\[
{\text{{\em 1-statement:}}}\quad\mu_p(\F)=1-o(1), \quad\text{if $p\gg p_c$}.
\]
In many situations, establishing the 0-statement is a much easier task than the 1-statement, as typically one can easily prove the existence of certain obstructions for $\F$, e.g.\ isolated vertices for $\F={\mathcal F}_{\texttt{PM}}$. The Kahn-Kalai conjecture states that
$p_c(\F)=O(p_E(\F) \log N)$, where $p_E(\F)$ is the so-called expectation threshold, which will be formally defined in Section~\ref{section: Uniform Spread Intro}. Roughly speaking, it can be viewed as the ``best lower bound'' that can be obtained by considering the expected number of obstructions.

We highlight two key advances on the Kahn-Kalai conjecture, which are restated as Theorem~\ref{Uniform Kahn-Kalai Theorem} and Theorem~\ref{Uniform Fractional Kahn-Kalai Theorem} respectively. First, Park and Pham~\cite{park2023proofkahnkalaiconjecture} confirmed the conjecture, which reduces determining thresholds for an increasing family to estimating the expectation threshold. However, the latter still remains challenging as it involves optimizing the ``covers'' of the given increasing family $\F$. Second, an earlier work by Frankston, Kahn, Narayanan, and Park~\cite{frankston2019thresholds} uses the dual of the cover optimization to provide an upper bound for the primal and thereby demonstrate the 1-statement. The dual variables have a combinatorial interpretation known as \vocab{spread}. Proving the 1-statement for $\F$ involves estimating the spread of structures related to $\F$. This approach has significantly advanced threshold determination for long-standing problems in random graph theory (such as the square of Hamilton cycles in~\cite{kahn2020thresholdsquarehamiltoncycle}, the bounded degree subgraphs in~\cite{frankston2019thresholds}, the Erd\H{o}s conjecture about high girth Steiner triple systems in~\cite{kwan2024highgirthsteinertriplesystems}, etc.) while also greatly simplifying proofs for existing theorems with very involved proofs such as the bounded degree spanning trees (originally proved by Montgomery in~\cite{montgomery2019spanningtreesrandomgraphs} and later proved using Kahn-Kalai in~\cite{frankston2019thresholds}).

In this paper, we discuss a few generalizations and variants of the Park-Pham theorem (i.e.\ the Kahn-Kalai conjecture). The Park-Pham theorem concerns a probability measure on the subsets of $X$ by including every element of $X$ with a uniform probability $p$. What happens if the inclusion probability is not uniform for the elements? 
In a recent work~\cite{park2023simple}, B.~Park and Vondr{\'a}k proved that the Park-Pham theorem holds for this generalized version; see the precise statement in Section~\ref{section: Non-Uniform Spread}.
As before, this theorem relates two thresholds and does not directly yield an estimate for the threshold of an increasing family $\F$. Indeed, the notion of ``threshold'' is not so clear any more, as we are dealing with  vectors $\vec{p}\in [0,1]^X$ rather than a single real probability $p$. We first extend the notion of thresholds and phase transitions in this more general setting.
Given $\vec{p}\in [0,1]^X$, let $\mu_{\vec{p}}$ be the product measure on $2^X$ where
\begin{equation}
    \mu_{\vec{p}}(S) := \prod\limits_{x\in S} p_x\prod\limits_{y\notin S} (1-p_y)\quad \text{for all $S\subseteq X$} \label{def:measure}
\end{equation}
 and $\mu_{\vec{p}}(\mathcal{F}) := \sum\limits_{S\in\mathcal{F}} \mu_{\vec{p}}(S)$. Let $X_{\vec{p}}$ be a random subset of $X$ drawn from $\mu_{\vec{p}}$. 
In the uniform case,  $p_c(\mathcal{F})$ is defined by $\mu^{-1}(1/2)$, where $\mu=\mu_{p}({\mathcal F})$ is viewed as a function of $p$, which is known to be continuous and increasing since ${\mathcal F}$ is an increasing family.

This notation of threshold naturally generalises to  the non-uniform setting, where we define $p_c(\mathcal{F})$ to be the following set of $\vec{p}$:
\begin{equation}
\{\vec{p}\in [0,1]^X: \mu_{\vec{p}}(\mathcal{F})=1/2\}. \label{def:critical}
\end{equation}
The theorem below demonstrates that phase transitions do occur around the values of $\vec{p}$ in the set~\eqref{def:critical}.

\begin{theorem}
\label{theorem: Non-Uniform Friedgut}
Let $\mathcal{F}\neq\emptyset, 2^X$ be an increasing family such that $\mathcal{F}\subseteq 2^X$ and $0<\eps<1$ be fixed. If $\vec{p}^*\in [0,1]^X$ such that $\eps<\mu_{\vec{p}^*}(\mathcal{F})<1-\eps $ and $\alpha = \omega(1)$, then
\begin{center}
    $\lim\limits_{\abs{X}\rightarrow\infty} \mu_{\vec{p}}(\mathcal{F}) = \begin{cases} 0 &\text{ if } \vec{p}\leq\vec{p}^*/\alpha\\ 1 &\text{ if } \vec{p}\geq\alpha\vec{p}^*\end{cases}$.
\end{center}
\end{theorem}

\begin{remark*}
\normalfont
Such $\vec{p}^*$ always exists. To see this, note that if $\mathcal{F}\neq\emptyset, 2^X$, then for each $0 < \delta < 1$, there exists $\vec{p}\in [0,1]^X$ such that $\mu_{\vec{p}}\left(\mathcal{F}\right) = \delta$. This is because $\mu_{\vec{p}}\left(\mathcal{F}\right)$ is a polynomial in $\vec{p}$, and an increasing function that attains both 0 (at $\vec{p}=\vec{0}$) and 1 (at $\vec{p}=\vec{1}$), as $\mathcal{F}$ is an increasing family.    
\end{remark*}

However, unlike in the uniform setting, given $\vec{p}$, it is not clear which $\vec{p}^*$ it should be compared to in order to determine which of the following cases $\vec{p}$ falls into.
\begin{itemize}
    \item Supercritical case: $\vec{p}\gg \vec{p}^*$ for some $\vec{p}^*\in p_c({\mathcal F})$;
    \item Subcritical case: $\vec{p}\ll \vec{p}^*$ for some $\vec{p}^*\in p_c({\mathcal F})$; 
    \item Near the threshold case: $\vec{p}=\Theta( \vec{p}^*)$ for some $\vec{p}^*\in p_c({\mathcal F})$;
    \item None of the above.
\end{itemize}

To address this issue, we introduce the \vocab{threshold map}, which enables us to determine whether $\vec{p}$ is supercritical, subcritical, near the threshold, and to show that essentially these are the only three regimes for $\vec{p}$.

 Given a sequence of probability spaces indexed by $n$, we say a sequence of events $A_n$ occurs asymptotically almost surely (a.a.s.) if $\lim_{n\to\infty} \prob{A_n}=1$.

\begin{definition}\label{def:threshold}
Suppose that $f_{\mathcal{F}}:[0,1]^X\to {\mathbb R}_{\ge 0}$ satisfies
\begin{itemize}
    \item[(a)] (Subcritical) If $f_{\mathcal{F}}(\vec{p})\ll 1$ then a.a.s.\ $X_{\vec{p}}\notin\mathcal{F}$.

    \item[(b)] (Supercritical) If $f_{\mathcal{F}}(\vec{p})\gg 1$ then a.a.s.\ $X_{\vec{p}}\in\mathcal{F}$.
\item[(c)] (Near the threshold)
$\forall 0<\eps<1\ \exists n_0>0, \delta>0,\ \forall n>n_0,\ \vec{p}\in [0,1]^{X}:$
$$
\eps<f_{\mathcal{F}}(\vec{p})<\eps^{-1} \implies \delta<\mu_{\vec{p}}(\mathcal{F})<1-\delta.
$$
\end{itemize}
If $f_{\mathcal{F}}$ satisfies (a,b) then we call it a {\em threshold map} of $\F$. If $f_{\mathcal{F}}$ satisfies (c) additionally, then we call it a {\em faithful threshold map} of $\F$.
\end{definition}

By the subsubsequence principle, we may assume  that $\lim_{n\to\infty} f_{\mathcal{F}}(\vec{p})$ exists (including the case that the limit is $\infty$). Hence, any $\vec{p}$ must lie in one of the cases (a,b,c). Moreover, the value $f_{\mathcal{F}}(\vec{p})$ determines precisely which regime $\vec{p}$ lies in.

Our next theorem assures that for an increasing family ${\mathcal F}$, a faithful threshold map exists.

\begin{theorem}
\label{theorem: Existence of Non-Uniform Threshold}
Let $\mathcal{F}\subseteq 2^X$ be an increasing family such that $\mathcal{F}\neq\emptyset, 2^X$. Then, there exists an increasing continuous faithful threshold map $f_{\mathcal{F}}$ of $\mathcal{F}$.
\end{theorem}

\proof 
Since $\mathcal{F}$ is an increasing family, $\mu_{\vec{p}}(\mathcal{F})$ is an increasing function of $\vec{p}$. Let  $f_{\mathcal{F}}:[0,1]^X\rightarrow\R_{\geq 0}$ be defined by
\begin{center}
    $f_{\mathcal{F}}(\vec{p}) = \begin{cases} \mu_{\vec{p}}(\mathcal{F}) &\text{ if } \mu_{\vec{p}}(\mathcal{F})\le 1/2\\ 
    (4(1-\mu_{\vec{p}}(\mathcal{F})))^{-1} &\text{ otherwise}\end{cases}$.
\end{center}
Obviously, $f_{\mathcal{F}}(\vec{p})$ is a continuous increasing function of $\mu_{\vec{p}}(\mathcal{F})$ and thus a continuous increasing function of $\vec{p}$. Moreover, $f_{\mathcal{F}}(\vec{p})$ satisfies all conditions (a,b,c) of Definition~\ref{def:threshold}.   \qed
\medskip

The construction of $f_{\mathcal{F}}$ is abstract and uses the function $\mu_{\vec{p}}({\mathcal F})$, which gives no information of $p_c({\mathcal F})$ other than its definition~\eqref{def:critical}. Hence, the significance of Theorem~\ref{theorem: Existence of Non-Uniform Threshold} lies in establishing the existence of a faithful threshold map, rather than in providing a method for constructing one.
In applications with given ${\mathcal F}$, we wish to find ``nice'' threshold maps that are as faithful as possible, which give estimations of members in $p_c({\mathcal F})$. For instance, in uniform setting $\G(n,p)$, and ${\mathcal F}={\mathcal F}_{\texttt{PM}}$, the map $f_1=pn/\log n$ is a threshold map, but not a faithful one, because the appearance of a perfect matching in $\G(n,p)$ has a sharp threshold. On the other hand, $f_2=\exp(pn-\log n-\log\log n)$ is a faithful threshold map.

However, finding ``nice'' $f_{\mathcal{F}}$ across the entire domain $[0,1]^X$ (as in Definition~\ref{def:threshold}) is often challenging, even for the simplest increasing families $\mathcal{F}$, due to the potentially complex shapes of $\vec{p}$ at thresholds. It is usually more feasible and informative to focus on a particular family ${\mathcal P}$ of probability vectors $\vec{p}$ and examine the threshold map within ${\mathcal P}$. Notable examples include the stochastic model and the Chung-Lu model (see Section~\ref{section: Random Graph Models} for further details), which have been extensively studied in the past. When working with a restricted domain ${\mathcal P}\subseteq [0,1]^X$, we naturally adapt the notions of threshold and threshold map by replacing $[0,1]^X$ with ${\mathcal P}$ in Definition~\ref{def:threshold}.

Our first contribution extends~\cite{frankston2019thresholds} in two directions. First, we develop the notion of \vocab{spread} for an increasing family $\mathcal{F}$ that is tailored to proving the 1-statement for $\mathcal{F}$ in $X_{\vec{p}}$. Our second extension considers $\mathcal{F}$ where every $S\in \mathcal{F}$ is forced to contain a subset of elements $F\subseteq X$. For example, consider $X_p$ being $\G(n,p)$ and $\mathcal{F}$ being the set of Hamiltonian graphs on $[n]$ that contain a particular set $F$ of edges. This setting allows many applications by patching smaller structures to form a desirable large structure. For instance, Frieze~\cite{HamiltonCyclePatching}  found Hamilton cycles by patching disjoint cycle covers or paths together using a set of specified edges. This technique has been used in many other works. Can the spread method be applied to find upper bounds on the threshold in this case? Our Theorem~\ref{Hamilton Spread Permutation Theorem} below aims to answer this question; refer to Section~\ref{section: Non-Uniform Spread} for the precise statement and the relevant definitions.

The second contribution of this paper is to test the strength of Theorem~\ref{Restatement of P-V Theorem 1} in capturing the threshold and the associated threshold map $f_{\mathcal{F}}$. Our approach relies on the notion of spread via Theorem~\ref{Analogue of Observation 7.5}, which provides a $1$-statement. This allows us to address questions such as: How powerful is this theorem? Does it capture the expectation threshold, and how accurately does it estimate the threshold for a given increasing family $\mathcal{F}$?
As a test case, we focus on the increasing family $\mathcal{F}=\mathcal{F}_{\texttt{PM}}$. In Section~\ref{section: Test Case with Perfect Matchings}, we confirm that Theorem~\ref{Analogue of Observation 7.5} indeed captures the threshold in a broad range of regimes. Nevertheless, it is worth emphasizing that even for this relatively simple example of $\mathcal{F}$, estimating $f_{\mathcal{F}}$
  over the full domain $[0,1]^{\binom{[n]}{2}}$, or even over more structured domains such as those arising from the stochastic block model or the Chung–Lu model, remains highly challenging, regardless of whether the spread method is applied.

We note that in some cases, applying Theorem~\ref{Hamilton Spread Permutation Theorem} together with the ``patching technique'' allows us to capture the thresholds,  when a direct application of Theorem~\ref{Analogue of Observation 7.5} introduces a logarithmic slack (this corresponds to the situations where the actual phase transition occurs close to the lower bound from the expectation threshold, rather than the Kahn-Kalai upper bound; see details of these two bounds in Section~\ref{section: Uniform Spread Intro}). Such cases do not arise in the study of $\mathcal{F}_{\texttt{PM}}$.
In a subsequent paper, we will extend this study to more difficult examples, including Hamilton cycles and spanning subgraphs of bounded degrees. For these examples, we will fully leverage the power of Theorems~\ref{Hamilton Spread Permutation Theorem} and~\ref{Analogue of Observation 7.5}.

Finally, we turn to ${\mathcal G}(n,\vec{d})$, the uniform random graph with prescribed degree sequence $\vec{d}=(d_1,\ldots,d_n)$.
This model is substantially more difficult to analyze than $X_{\vec{p}}$ with $X=K_n$, or other models such as the stochastic model or the Chung-lu model. Nevertheless, it is among the most important random graph models due to its wide range of applications.

 Suppose that $\vec{d}$ and $\vec{d}'$ are two degree sequences where $\vec{d}\le \vec{d}'$. Unlike $X_{\vec{p}}$, it is no longer always true that ${\mathcal G}(n,\vec{d})$ and ${\mathcal G}(n,\vec{d}')$ can be coupled so that ${\mathcal G}(n,\vec{d})$ is a subgraph of ${\mathcal G}(n,\vec{d}')$. As a result, thresholds (now in terms of $\vec{d}$ instead of $\vec{p}$) need not always exist for an increasing family $\F$.
Our last contribution in this paper is to bridge Kahn-Kalai with $\mathcal{G}(n,\vec{d})$. Because the edges do not appear independently in $\mathcal{G}(n,\vec{d})$, Theorem~\ref{Analogue of Observation 7.5} cannot be applied immediately to this model. However, we use a result from~\cite{Gao_2024} to embed $X_{\vec{p}}$ within $\mathcal{G}(n,\vec{d})$ for some properly chosen $\vec{p}$. This allows us to apply Theorem~\ref{Analogue of Observation 7.5} for $X_{\vec{p}}$ and obtain results for $\mathcal{G}(n,\vec{d})$; see section~\ref{section:G(n,d) results} for more details.

\subsection{Inhomogeneous random graph models}\label{section: Random Graph Models}

With $X=\binom{[n]}{2}$, $X_{\vec{p}}$ defines a random graph on $[n]$. However, it is more conventional in this case to denote the random graph by $\G(n,P)$, where $P\in [0,1]^{\binom{[n]}{2}}$. Although determining thresholds over the entire domain $[0,1]^{\binom{[n]}{2}}$
is highly complex, many interesting and tractable cases arise by restricting $P$ to some structured subfamilies $\P$. In this section, we highlight several well-studied random graph models corresponding to such subfamilies, and in the following sections we explore thresholds within ${\mathcal P}$.

We start with the stochastic block model in which $P$ respects a specified block structure. The stochastic block model appeared in~\cite{SBM_Intro}, and there are many research works relating to community analysis/detection in networks, such as in~\cite{SBM_First_Community_Detection} and~\cite{SBM_Second_Community_Detection}. 

\begin{definition*} Let $k\ge 1$ be a positive integer, $\bfU=(U_1,\ldots, U_k)$ be a set of $k$ pairwise disjoint finite sets, and $P\in [0,1]^{k\times k}$ be a symmetric matrix. 
The stochastic block model $\mathcal{B}(\bfU, P)$ consists of vertex set $U_1\cup\cdots\cup U_k$ in which two distinct vertices $u\in U_i$ and $v\in U_j$ are adjacent independently with probability $P_{ij} = P_{ji}$, for all $1\leq i,j\leq k$. Each $U_i$ is called a block of vertices in the model.
\end{definition*}

\begin{remark*}
\normalfont
In most works, the stochastic model is defined so that $k$ is fixed and is independent of the total number of vertices in the graph. We do not include that restriction in the definition, and hence $\G(n,P)$ can be viewed as a stochastic block model by letting $U_i=\{i\}$ for $i\in [n]$. However, in  most interesting examples that we investigate, $k$ will be fixed.
\end{remark*}

Next, we introduce the \vocab{Chung-Lu model} in which $P$ is ``guided'' by the so-called ``expected degree sequence''.

\begin{definition*}
Given a sequence $\vec{d} = (d_1,\ldots,d_n)$ of non-negative integers, called the expected degree sequence, the \vocab{Chung-Lu model} is given by $\mathcal{G}_{\text{CL}}(n,\vec{d}) := \mathcal{G}(n, P_{\vec{d}})$ where $P_{\vec{d}}$ is the edge probability matrix satisfying
\begin{center}
    $P_{\vec{d}}(i,j) = \min\left\{1, \displaystyle\frac{d_i d_j}{\norm{\vec{d}}_1}\cdot\mathbbm{1}\{i\neq j\}\right\}$
\end{center}
for all $1\leq i,j\leq n$.
\end{definition*}

\begin{remark*}
\normalfont
The Chung-Lu model can be viewed as a special case of the stochastic block model, where the vertices associated with the same expected degree are contained in the same block.
Without loss of generality, we may assume that $d_1\geq\cdots\geq d_n\geq 1$ (if $d_i=0$ then one can safely consider the equivalent model by removing vertex $i$). 
\end{remark*}

Finally, we define  $\mathcal{G}(n,\vec{d})$. Note that the probability measure of $\G(n,\vec{d})$ does not follow~(\ref{def:measure}) due to the edge dependency.
    
\begin{definition*}
Given a sequence $\vec{d} = (d_1,\ldots, d_n)$ of nonnegative integers, $\mathcal{G}(n,\vec{d})$ is a uniformly random graph where vertex $i$ has degree $d_i$ for every $1\le i\le n$.
\end{definition*}

\subsection{Test case: Perfect matchings}\label{section: Test Case with Perfect Matchings}

By extending the notion of spread in \cite{frankston2019thresholds}, Theorem~\ref{Analogue of Observation 7.5} (in Section~\ref{section: Non-Uniform Spread} below) offers a method to transform the problem of identifying thresholds in the non-uniform setting into a purely combinatorial problem. Due to technical terms and definitions, we defer the presentation of Theorem~\ref{Analogue of Observation 7.5} and related terms to Section~\ref{section: Non-Uniform Spread}.

In this section, we first discuss the applications of Theorem~\ref{Analogue of Observation 7.5}. As commented before, determining the threshold map $f_{\mathcal{F}}$ can be very difficult, and we may not even expect good analytic properties $f_{\mathcal{F}}$, even for the simplest examples of $\mathcal{F}$. To test the effectiveness of Theorem~\ref{Analogue of Observation 7.5}, this section focuses on one of the simplest increasing families consisting of graphs containing a perfect matching.

Recalling $\mathcal{B}(\bfU,P)$ from Section~\ref{section: Random Graph Models}, we aim to derive conditions on $(\vec{U},P)$ for which $\mathcal{B}(\vec{U},P)$ has a perfect matching a.a.s., as well as conditions where it does not. Given $\vec{U}$, let $\abs{\vec{U}}$ denote the number of parts in $\bfU$, i.e.\ the integer $k$ such that $\vec{U}=(U_1,\ldots,U_k)$. To state our result, we define a \vocab{graph spectrum} $\mathfrak{G}_{\texttt{PM}}^\alpha(\vec{U}, P)$ that is associated to $\mathcal{B}(\vec{U},P)$ as follows.

\begin{definition}
\label{defn: graph spectrum}
Given $\mathcal{B}(\vec{U},P)$, and  $\alpha\geq 0$, let $\mathfrak{G}_{\texttt{PM}}^\alpha(\vec{U}, P)$ be the graph with vertex set $\bigcup\limits_{j=1}^{\abs{\vec{U}}} U_j$ such that two vertices $u\in U_x$ and $v\in U_y$ are adjacent if and only if $P(x,y)\geq \alpha\log n/\max(n_x, n_y)$, where $n_j= \abs{U_j}$ for all $1\leq j\leq\abs{\vec{U}}$. 
\end{definition}

Our $1$-statement for perfect matchings in $\mathcal{B}(\vec{U},P)$ is associated with the existence of a perfect matching in $\mathfrak{G}_{\texttt{PM}}^\alpha(\vec{U}, P)$ for an appropriate $\alpha$.

\begin{theorem}
\label{theorem: General Stochastic Perfect Matching Improved}
If $\mathfrak{G}_{\texttt{PM}}^{\alpha} (\vec{U}, P)$ has a perfect matching for some $\alpha = \omega(1)$, then a.a.s.\  $\mathcal{B}(\vec{U},P)$ has a perfect matching. 
\end{theorem}

Note that $|\vec{U}|$ does not have to be fixed in the above theorem. Indeed, we will apply Theorem~\ref{theorem: General Stochastic Perfect Matching Improved} to the Chung-Lu model later in this section, where the number of blocks can be a function of $n$.

An ideal $0$-statement would be the following:

\begin{mdframed}[backgroundcolor=gray!30]
\textbf{Ideal 0-statement}: If  $\mathfrak{G}_{\texttt{PM}}^\alpha (\vec{U}, P)$ does not have a perfect matching for some $\alpha = o(1)$, then a.a.s.\ $\mathcal{B}(\bfU,P)$ does not have a perfect matching.
\end{mdframed}

If the above statement were true, then this demonstrates a phase transition because we can take $f_{\texttt{PM}}(\vec{U}, P)$ to be the maximum $\alpha$ for which $\mathfrak{G}_{\texttt{PM}}^{\alpha}(\vec{U}, P)$ has a perfect matching (noting that $\mathfrak{G}_{\texttt{PM}}^{\alpha_1}(\vec{U},P)\subseteq\mathfrak{G}_{\texttt{PM}}^{\alpha_2}(\vec{U},P)$ if $\alpha_1\geq\alpha_2$). If $\mathfrak{G}_{\texttt{PM}}^{\alpha}(\vec{U},P)$ has no perfect matching for some $\alpha = o(1)$, then $f_{\texttt{PM}}(\vec{U}, P)\ll 1$. Together with Theorem~\ref{theorem: General Stochastic Perfect Matching Improved}, this function $f_{\texttt{PM}}(U, P)$ is a threshold map. However, the ideal 0-statement above is false in general; we give a counterexample below. The proof is left as an exercise to the reader, as it follows closely the argument of Example~\ref{example: G(n,d) Ideal 0-statement Example}, whose proof we provide. 

\begin{example}
\label{example: SBM Ideal 0-statement Counterexample}
Consider $\vec{U} = (U_1, U_2)$ with $\abs{U_2} = \abs{U_1}+2$ and $P\in [0,1]^{n\times n}$ such that $P(x,y) = 1$ for all $x\in U_1$ or $y\in U_1$. There exists $\rho = o(\log n/n)$ such that if $P(x,y) = \rho$ for all $x,y\in U_2$, then $\mathcal{B}(\vec{U}, P)$ has a perfect matching with positive probability as $n\rightarrow\infty$. However, note that for some $\alpha = o(1)$, the graph $\mathfrak{G}_{\texttt{PM}}^{\alpha} (\vec{U}, P)$ has no edges between any two vertices in $U_2$, and therefore, has no perfect matching.
\end{example}

Since this $0$-statement is not true in general, it is natural to ask the following question:
\begin{problem}
\label{Perfect Matching $0$-statement}
Can we define (interesting) families of $(\vec{U},P)$ for which the ideal 0-statement is true?
\end{problem}

We will address Problem~\ref{Perfect Matching $0$-statement} for some families of $(\vec{U},P)$ arising from the Chung-Lu model. First, we focus on the Chung-Lu model where the expected degree sequence has a fixed number of values.

\begin{definition}
\label{defn: U and P for Chung-Lu}
We say a sequence $\vec{d} = (d_1,\ldots, d_n)$ is \vocab{$k$-valued} if there exist $k$ distinct real numbers $\alpha_1,\ldots,\alpha_k$ such that $d_i\in \{\alpha_1,\ldots,\alpha_k\}$ for all $1\le i\le n$.
\end{definition}

Without loss of generality we may assume that the components in a $k$-valued sequence $\vec{d}$ take exactly $k$ values, as otherwise, we may consider it a $k'$-valued sequence for some $k'<k$.
Let $U_i = \{v_j: d_j = \alpha_i\}$ and $n_i = \abs{U_i}$ for all $1\leq i\leq k$. 
Furthermore, let $\vec{U}_{\vec{d}} = \left(U_1,\ldots,U_k\right)$,  which is a partition of the vertices, and let $P_{\vec{d}}$ be the probability matrix inherited from $\mathcal{G}_{\text{CL}}(n, \vec{d})$. It is immediate that $\mathcal{G}_{\text{CL}}(n, \vec{d}) = \mathcal{B}(\vec{U}_{\vec{d}}, P_{\vec{d}})$. Our first focus is on bi-valued (i.e., $2$-valued) sequences $\vec{d} = (\underbrace{d_1,\ldots, d_1}_{n_1}, \underbrace{d_2,\ldots, d_2}_{n_2})$, with $d_1> d_2$. Define families $\mathcal{D}_1$ and $\mathcal{D}_2$ of bi-valued sequences as follows. 
\begin{itemize}
    \item $\vec{d}\in\mathcal{D}_1$ if $n_1\geq n_2 \ge n^{\Theta(1)}$;
    \item $\vec{d}\in\mathcal{D}_2$ if $n_1\leq n^{\delta}$ for some fixed $\delta < 1$.
\end{itemize}

We prove that the ideal 0-statement holds for bi-valued sequences $\vec{d}$ in $\mathcal{D}_1\cup\mathcal{D}_2$, and therefore, the graph spectrum $\mathfrak{G}_{\texttt{PM}}^\alpha(\vec{U}_{\vec{d}}, P_{\vec{d}})$ correctly captures the threshold for perfect matchings for these families of degree sequences.

\begin{theorem}
\label{theorem: Phase Transition for PM in Bi-Valued Families}
The ideal 0-statement holds in $\mathcal{G}_{\text{CL}}(n, \vec{d})$, for $\vec{d}\in\mathcal{D}_1\cup\mathcal{D}_2$.
\end{theorem}

We give an example below which can be viewed as an extension of $\mathcal{D}_1$ from the bi-valued case to the $k$-valued case; noting that $k$  may depend on $n$. Refer to section~\ref{section: Proof of k-valued Chung-Lu} for its proof. 

\begin{example}
\label{example: 0-statement Chung-Lu k-valued}
 Let $\mathcal{D}$ be a family of $k$-valued sequences $\vec{d}$ such that $n_1\geq\frac{n}{2}$ and $n_k \ge n^{\delta}$ for some fixed $\delta>0$. Then, the ideal 0-statement holds in $\mathcal{G}_{\text{CL}}(n,\vec{d})$ for $\vec{d}\in\mathcal{D}$.    
\end{example}

In light of Theorem~\ref{theorem: Phase Transition for PM in Bi-Valued Families} and Example~\ref{example: 0-statement Chung-Lu k-valued}, we propose the following problem for characterising $k$-valued sequences for which the ideal 0-statement holds. The problem is already challenging even for fixed $k$.

\begin{problem} Let $k\ge 2$ be a fixed integer.
Characterize the $k$-valued sequences for which the ideal $0$-statement holds.
\end{problem}

\subsection{Between non-uniform Kahn-Kalai and \pdfmath{\mathcal{G}(n,\vec{d})}}
\label{section:G(n,d) results}

Theorem~\ref{theorem: General Stochastic Perfect Matching Improved} cannot be applied directly to guarantee the existence of a perfect matching in $\G(n,\vec{d})$, since the model $\G(n,\vec{d})$  does not fit into the stochastic block model framework; in particular, the edges no longer appear independently. In fact, $\G(n,\vec{d})$ is a much more challenging model to analyze, and much of the analysis depends heavily on enumeration techniques. Moreover, the Kahn–Kalai upper bound for the emergence of perfect matchings (and other structures), while tight in asymptotic order for $\G(n,p)$, may not be tight in $\G(n,\vec{d})$ any more. A simple illustration of that comes from the well-known result that the random $d$-regular graph has a perfect matching a.a.s.\ provided that $d\ge 3$.

This motivates our study of $\G(n,\vec{d})$. In particular, we are interested in understanding whether an analogue of Theorem~\ref{theorem: General Stochastic Perfect Matching Improved} holds for  $\G(n,\vec{d})$, after imposing some reasonable conditions on $\vec{d}$, and if so, whether it captures the threshold up to the correct asymptotic order. We investigate these questions in this section. 

Our first result is to confirm that an analogue of Theorem~\ref{theorem: General Stochastic Perfect Matching Improved} holds for  $\G(n,\vec{d})$ for a large family of degree sequences $\vec{d}$.  Owing to a set of reasons, some obvious and others technical (refer to Section~\ref{section: proof of G(n,d) 1-statement} for more details), we restrict our discussions to sequences $\vec{d} = (d_1,\ldots,d_n)$ (with $d_1\geq d_2\geq\cdots \ge d_n$) such that 
\begin{equation}\label{eqn: deg-condition}
    \norm{\vec{d}}_1=0\pmod{2},\quad d_1^2 = o(\norm{\vec{d}}_1),\quad \text{and}\quad d_n\gg\log n.
\end{equation}
Recall $\vec{U}_{\vec{d}}$, $P_{\vec{d}}$ that are defined below Definition~\ref{defn: U and P for Chung-Lu}, and $\mathcal{G}_{\texttt{PM}}^{\alpha}(\vec{U}_{\vec{d}}, P_{\vec{d}})$ in Definitions~\ref{defn: graph spectrum}.
\begin{theorem}
\label{theorem: G(n,d) 1-statement}
Suppose that $\vec{d}$ is a degree sequence satisfying~\eqref{eqn: deg-condition}. If $\mathfrak{G}_{\texttt{PM}}^{\alpha}(\vec{U}_{\vec{d}}, P_{\vec{d}})$ has a perfect matching for some $\alpha = \omega(1)$, then a.a.s.\ $\mathcal{G}(n,\vec{d})$ has a perfect matching. 
\end{theorem}

Next 
we investigate whether Theorem~\ref{theorem: G(n,d) 1-statement} is tight for degree sequences satisfying~\eqref{eqn: deg-condition}; i.e.\ whether the ideal 0-statement (restricted to degree sequences satisfying~\eqref{eqn: deg-condition}) holds. It is not a surprise that this is not the case. We give a negative example below, which is an analog of Example~\ref{example: SBM Ideal 0-statement Counterexample} (refer to section~\ref{section: Proof of G(n,d) Ideal 0-statement Counterexample} for the proof). 

\begin{example}
\label{example: G(n,d) Ideal 0-statement Counterexample}
Let $\vec{d}$ be bi-valued such that $n_2 = n_1 + 2$, $d_1 = \sqrt{n}$, and $d_2 = n^{1/8}\sqrt{\log n}$. Then $\mathfrak{G}_{\texttt{PM}}^{\alpha}(\vec{U}_{\vec{d}}, P_{\vec{d}})$ has no perfect matching for some $\alpha = o(1)$, but a.a.s.\ $\mathcal{G}(n,\vec{d})$ has a perfect matching. 
   
\end{example}

We naturally come to the next question: are there sub-families of degree sequences satisfying~\eqref{eqn: deg-condition} such that the ideal 0-statement holds? Our example below demonstrates that such sub-families exist (refer to section~\ref{section: Proof of G(n,d) concrete threshold} for the proof). 

\begin{example}
\label{example: G(n,d) Ideal 0-statement Example}
Let $\vec{d}$ be bi-valued such that $d_1 = n^{1/8}$ and $n_1 = n^{15/16}$. Then
\begin{center}
    $\lim\limits_{n\rightarrow\infty}\prob{\mathcal{G}(n,\vec{d}) \text{ has a perfect matching}} = \begin{cases} 0 &\text{, if } d_2\ll n^{1/32}\sqrt{\log n}\\ 1 &\text{, if } d_2\gg n^{1/32}\sqrt{\log n}\end{cases}$.
\end{center}
\end{example}

\begin{remark*}
The phase transition in Example~\ref{example: G(n,d) Ideal 0-statement Example}, expressed in terms of $d_2$, corresponds precisely to the dichotomy between the case where $\mathfrak{G}_{\texttt{PM}}^\alpha(\vec{U}_{\vec{d}},P_{\vec{d}})$ has a perfect matching for some $\alpha = \omega(1)$ (1-statement) and the case where it does not have a perfect matching for some $\alpha = o(1)$ (0-statement).
\end{remark*}

In the following theorem, we broaden the family of degree sequences in the above example. 

\begin{theorem}
\label{theorem: G(n,d) bi-valued 0-statement}
      Let $\mathcal{D}$ be the family of bi-valued $\vec{d}$  with $\norm{\vec{d}}_1=0\pmod{2}$ and the following properties:
      \begin{enumerate}[label=(A\arabic*), leftmargin=3cm]
          \item $d_2^2 = o(n_1)$ \label{cond: d_2 and n_1 bound}
          \item $d_1^2 = o(\sqrt{n d_2})$ \label{cond: stronger d_1 bound}
          \item $d_2\gg\log n$ \label{cond: d_2 lower bound embedding}
          \item $n_2^2 d_2^3 = o(\norm{\vec{d}}_1^2)$ \label{cond: n_2 and d_2 bound with norm}
          \item $n_1\leq n^\delta$ for some fixed $\delta < 1$ \label{cond: n_1 small}
      \end{enumerate}
      Then, both the 1-statement and ideal 0-statement hold for $\mathcal{D}$.
\end{theorem}

\begin{remark}
\normalfont
\label{remark: G(n,d) remark}
Note that condition~\ref{cond: stronger d_1 bound} implies $d_1^2 = o(\norm{\vec{d}}_1)$ since $nd_2\leq\norm{\vec{d}}_1$. Therefore, $\mathcal{F}$ indeed only contains degree sequences satisfying~\eqref{eqn: deg-condition}. Furthermore, we do not need to assume $d_2\gg\log n$ for the ideal 0-statement; it is only a technical condition used for the 1-statement (i.e., Theorem~\ref{theorem: G(n,d) 1-statement}).
\end{remark}

Finally, we propose the following problem.

\begin{problem} 
Characterize the degree sequences $\vec{d}$ for which Theorem~\ref{theorem: G(n,d) 1-statement} is tight.
\end{problem}

\subsection{Threshold, integer program, and spread}
\label{section: Uniform Spread Intro}

Recall from the earlier introduction that the Park-Pham theorem (i.e.\ the Kahn-Kalai conjecture) asserts that $p_c(\mathcal{F})$ is bounded between $q(\mathcal{F})$, the so-called expectation-threshold, and $q(\mathcal{F})\log |X|$. Now we formally define $q(\mathcal{F})$ in the uniform setting, and we will generalize this notion to the non-uniform setting in Section~\ref{section: Non-Uniform Spread}. We try to keep the notation consistent with~\cite{park2023proofkahnkalaiconjecture}.

\begin{definition*}
We say $\mathcal{F}$ is \vocab{$q$-small} if there exists $\mathcal{G}\subseteq 2^X$ such that
\begin{center}
    $\mathcal{F}\subseteq\gen{\mathcal{G}}:= \bigcup\limits_{S\in\mathcal{G}} \{T: T\supseteq S\}$
\end{center}
and $\sum\limits_{S\in\mathcal{G}} q^{\abs{S}}\leq \frac{1}{2}$. The \vocab{expectation-threshold} of $\mathcal{F}$, $q(\mathcal{F})$, is the maximum $q$ for which $\mathcal{F}$ is $q$-small.
\end{definition*}

If $\mathcal{F}\subseteq\gen{\mathcal{G}}$, then we say that $\mathcal{G}$ is a \vocab{cover} for $\mathcal{F}$. Finding $q(\mathcal{F})$ can be expressed as an integer program in the following manner: 
\begin{alignat*}{7}
    V(\mathcal{F},q) 
    = \min& &&\displaystyle\sum\limits_{S\subseteq X} g(S) q^{\abs{S}}
    \\
    \text{subject to}&\quad &&\sum\limits_{S\subseteq T} g(S)\geq 1 &&\hspace{3mm}&&\forall\hspace{1mm}  T\in\mathcal{F}
    \\
    &&& g(S)\in\{0,1\} &&\hspace{3mm}&&\forall\hspace{1mm} S\subseteq X.
\end{alignat*}
It follows that $q(\mathcal{F})\leq q$ provided that $V(\mathcal{F}, q)\geq 1/2$. 
We refer the reader to~\cite{perkins2024searchingsharpthresholdsrandom} for a more detailed exposition of it.

Recalling that $p_c(\mathcal{F})$ is the unique $p$ for which $\mu_p(\mathcal{F}) = 1/2$ and that $\mu_p(\mathcal{F})$ is increasing in $p$, we can restate the main result of~\cite{park2023proofkahnkalaiconjecture} as follows.
 
\begin{theorem}[Park-Pham~\cite{park2023proofkahnkalaiconjecture}]
\label{Uniform Kahn-Kalai Theorem}
There exists a universal constant $K$ such that for every finite set $X$ and an increasing family $\mathcal{F}\subseteq 2^X$, if $q\in [0,1]$ satisfies $V(\mathcal{F},q)\geq 1/2$, then $\mu_p(\mathcal{F})\leq Kq\log\ell(\mathcal{F})$. 
\end{theorem}

While Theorem~\ref{Uniform Kahn-Kalai Theorem} beautifully describes the relation between $p_c({\mathcal F})$ and $q({\mathcal F})$, it is still very hard to find good bounds, especially good upper bounds, on $q({\mathcal F})$, for a given ${\mathcal F}$. A powerful tool called the spread method was introduced in~\cite{frankston2019thresholds} to bound $q({\mathcal F})$ from above, by investigating the fractional version of  $q({\mathcal F})$, which was first introduced in~\cite{TalagrandLPDuality}. We demonstrate this approach below.

We start from the linear relaxation of $V({\mathcal F},q)$ given below:
\begin{alignat*}{7}
    V_f(\mathcal{F},q) 
    = \min& &&\displaystyle\sum\limits_{S\subseteq X} g(S) q^{\abs{S}}
    \\
    \text{subject to}&\quad &&\sum\limits_{S\subseteq T} g(S)\geq 1 &&\hspace{3mm}&&\forall\hspace{1mm}  T\in\mathcal{F}
    \\
    &&& g(S)\geq 0 &&\hspace{3mm}&&\forall\hspace{1mm} S\subseteq X.
\end{alignat*}

This motivates the factional version of expectation-threshold given as follows.
\begin{definition*}
We say $\mathcal{F}$ is \vocab{weakly $q$-small} if there exists a function $g$ from $2^X$ to the non-negative reals $\R_{\geq 0}$ such that
\begin{center}
    $\displaystyle\sum\limits_{S\subseteq T} g(S)\geq 1$
\end{center}
for all $T\in\mathcal{F}$ and $\sum\limits_{S\subseteq X} g(S) q^{\abs{S}}\leq \frac{1}{2}$. The \vocab{fractional expectation-threshold} of $\mathcal{F}$, $q_f(\mathcal{F})$, is the maximum $q$ for which $\mathcal{F}$ is weakly $q$-small.
\end{definition*}

Obviously, $V({\mathcal F},q)\ge V_f({\mathcal F},q)$, and therefore if $V_f(\mathcal{F},q)\geq 1/2$, then $q_f(\mathcal{F})\leq q$. So, the upper bound of Theorem~\ref{Uniform Kahn-Kalai Theorem} immediately implies the following weaker version. (Indeed, this weaker version was proved before Theorem~\ref{Uniform Kahn-Kalai Theorem}.)

\begin{theorem}[Frankston et al~\cite{frankston2019thresholds}]
\label{Uniform Fractional Kahn-Kalai Theorem}
There exists a universal constant $K$ such that for every finite set $X$ and an increasing family $\mathcal{F}\subseteq 2^X$, if $q\in [0,1]$ satisfies $V_f(\mathcal{F},q)\geq 1/2$, then $\mu_p(\mathcal{F})\leq Kq\log\ell(\mathcal{F})$.
\end{theorem}

The main benefit from Theorem~\ref{Uniform Fractional Kahn-Kalai Theorem} is a relation between $V_f({\mathcal F},q)$ and the notion {\em spread}, which corresponds to the variables in the dual of $V_f({\mathcal F},q)$, given as follows.
\begin{alignat*}{7}
    &\max &&\displaystyle\sum\limits_{T\in\mathcal{F}} \nu(T)
    \\
    &\text{ subject to }\quad &&\nu(T)\geq 0 &&\hspace{3mm}&&\forall\hspace{1mm}  T\in\mathcal{F}
    \\ 
    & &&\sum\limits_{T\supseteq S} \nu(T)\leq q^{\abs{S}} &&\hspace{3mm}&&\forall\hspace{1mm} S\subseteq X.
\end{alignat*}

From this, the notion of spread is defined by adding an additional factor of $2$ to allow probability measures.

\begin{definition*}
Let $\mathcal{H}\subseteq 2^X$. A probability measure $\nu$ supported on $\mathcal{H}$ is \vocab{$q$-spread} if
\begin{center}
    $\displaystyle\sum\limits_{T\supseteq S} \nu(T)\leq 2q^{\abs{S}}$
\end{center}
for all $S\subseteq X$.
\end{definition*}

The main idea with spread is as follows. If $\nu$ is a $q$-spread probability measure supported on an increasing family $\mathcal{F}$, then taking $\nu' := \frac{1}{2}\nu$, it follows that $\nu'$ is a feasible solution of the dual of $V_f(\mathcal{F},q)$ with optimal value
\begin{center}
    $\displaystyle\sum\limits_{T\in\mathcal{F}} \nu'(T) = \frac{1}{2}\sum\limits_{T\in\mathcal{F}} \nu(T) = \frac{1}{2}$.
\end{center}
Then, by weak-duality, it follows that $V_f(\mathcal{F},q)\geq 1/2$ and we can apply Theorem~\ref{Uniform Fractional Kahn-Kalai Theorem}. In other words, the problem shifts to finding $q$-spread probability measures on increasing families.

Our applications specifically focus on subgraph containment problems, which, for the relevant choice of $\mathcal{F}$, will typically produce uniform hypergraphs $\mathcal{H}$ such that $\mathcal{F} = \gen{\mathcal{H}}$ (for simple graphs, $\mathcal{H}$ will simply be a $2$-uniform hypergraph). To that end, for $r\geq 1$, define an \vocab{$r$-graph} to be a hypergraph where the size of each edge is equal to $r$ and let $\mathcal{K}_n^r$ be the complete $r$-graph on $[n]$. Often in our applications, such as determining the threshold for the appearance of a perfect matching or a Hamiltonian cycle, we will want to take a specific $r$-graph $H\subseteq\mathcal{K}_n^r$ (where $\abs{V(H)} = n$) and look at the set of copies of $H$. To do this, let $S_n$ denote the symmetric group on $[n]$ and $\mathcal{G}_H := \{\sigma(H): \sigma\in S_n\}$ where $\sigma(H)$ denotes the graph obtained from $\sigma$ acting on $V(H)$. We adopt the convention that a group action on the vertex set of a graph induces an action on the edge set by acting on the corresponding subsets of vertices. Observation 7.5 in \cite{frankston2019thresholds}, which is restated below, determines all values of $q$ such that $\mathcal{G}_H$ supports a $q$-spread probability measure.

\begin{observation}[Frankston et al.~\cite{frankston2019thresholds}]
\label{Observation 7.5 in Fractional Paper}
Let $H\subseteq\mathcal{K}_n^r$, and let $\sigma\in S_n$ be a uniformly random permutation. Uniform measure on $\mathcal{G}_H$ is $q$-spread if and only if 
\begin{center}
    $\prob{\sigma(S)\subseteq H}\leq q^{\abs{S}}$
\end{center}
for all $S\subseteq\mathcal{K}_n^r$ isomorphic to a subhypergraph of $H$. 
\end{observation}

For our applications, we generally consider the previous results (these were stated more generally for any finite set $X$) for the case where $X$ is equal to $\binom{[n]}{r}$, the set of edges in $\mathcal{K}_n^r$. Observation~\ref{Observation 7.5 in Fractional Paper} can be used in applications $\mathcal{G}(n,p)$ where the edge probability is uniform.

\subsection{Non-uniform spread} 
\label{section: Non-Uniform Spread}

We extend the framework and definitions from~\cite{frankston2019thresholds} and Section~\ref{section: Uniform Spread Intro} to allow multiple probabilities. As before, let $X$ be a finite set and consider an increasing family $\mathcal{F}\subseteq 2^X$ with $\mathcal{F}\neq\emptyset, 2^X$. Given $\vec{p}\in [0,1]^X$, recall from~\eqref{def:measure} that $\mu_{\vec{p}}$ is the product measure on $2^X$ defined by
\begin{center}
    $\mu_{\vec{p}}(S) := \prod\limits_{x\in S} p_x\prod\limits_{y\notin S} (1-p_y)$
\end{center}
for all $S\subseteq X$ and that $X_{\vec{p}}$ is a random subset of $X$ drawn from $\mu_{\vec{p}}$. The following definitions are from \cite{park2023simple}. Given $\mathcal{G}\subseteq 2^X$ and $\vec{q}\in [0,1]^X$, define:
\begin{eqnarray}
\gen{\mathcal{G}} &=& \bigcup\limits_{S\in\mathcal{G}} \{T: T\supseteq S\}\\
\mu_{\vec{p}}(\mathcal{G}) &=& \displaystyle\sum\limits_{S\in\mathcal{G}} \mu_{\vec{p}}(S)\\
    e_{\vec{q}}(\mathcal{G}) &=& \expect{\abs{\{S\in\mathcal{G}: X_{\vec{q}}\supseteq S\}}} = \displaystyle\sum\limits_{S\in\mathcal{G}} \prod\limits_{x\in S} q_x\label{defn: e_q}.
\end{eqnarray}

\noindent That is, $e_{\vec{q}}(\mathcal{G})$ is the expected number of elements of $\mathcal{G}$ that are subsets of $X_{\vec{q}}$. Given $\vec{q}\in [0,1]^X$, we can define $V(\mathcal{H}, \vec{q}) := \min\{e_{\vec{q}}(\mathcal{G}):\ \mathcal{H}\subseteq \gen{\mathcal{G}}\}$. We say $\mathcal{H}\subseteq 2^X$ is \vocab{$\ell$-bounded} if $\abs{S}\leq\ell$ for all $S\in\mathcal{H}$.  Park and Vondr{\'a}k~\cite{park2023simple} generalized the Park-Pham Theorem to allow for non-uniform probabilities as follows.

\begin{theorem}[Park-Vondr{\'a}k~\cite{park2023simple}]
\label{Restatement of P-V Theorem 1}
Let $\ell\geq 1$, $\vec{q}\in [0,1]^X$ and $p_x = 1-(1-q_x)^{4\floor{\log(2\ell)}+7}$ for all $x\in X$. Let $\mathcal{H}\subseteq 2^X$ be an $\ell$-bounded family. If $V(\mathcal{H},\vec{q})\geq 1/2$, then $\mu_{\vec{p}}(\gen{\mathcal{H}})\geq 1/2$.
\end{theorem} 

\begin{remark*}
The original statement in~\cite{park2023simple}
(they used $c_{\vec{q}}(\mathcal{H})$ instead of $V(\mathcal{H},\vec{q})$) has strict inequalities (that is, if $V(\mathcal{H},\vec{q})> 1/2$, then $\mu_{\vec{p}}(\gen{\mathcal{H}})> 1/2$). Theorem~\ref{Restatement of P-V Theorem 1} follows immediately by the continuity of $\mu$. Indeed, suppose $\vec{q}\in [0,1]^X$ such that $V(\mathcal{H},\vec{q}) = 1/2$ and let $\mathcal{G}\subseteq 2^X$ such that $e_{\vec{q}}(\mathcal{G}) = V(\mathcal{H},\vec{q})$. Note that $\emptyset\notin\mathcal{G}$, as otherwise, $V(\mathcal{H},\vec{q}) = 1$. It follows that for any $\varepsilon > 0$, we have $V(\mathcal{H},\vec{q}_\varepsilon)\geq (1+\varepsilon)V(\mathcal{H},\vec{q}) > 1/2$ where $\vec{q}_\varepsilon := (1+\varepsilon)\vec{q}$ (since every term in the sum in~\eqref{defn: e_q} will be multiplied by $(1+\varepsilon)^{\abs{S}}\geq 1+\varepsilon$. If $\vec{p}_\varepsilon = T_\ell(q_\varepsilon)$ (see~\eqref{def:T} for the definition of $T_\ell$), then $\mu_{\vec{p}_\varepsilon}(\gen{H}) > 1/2$. By continuity of $\mu$ and since $\vec{p}_\varepsilon\rightarrow\vec{p}$ as $\varepsilon\rightarrow 0$, it follows that $\mu_{\vec{p}}(\gen{H})\geq 1/2$.
\end{remark*}

Observe that $\mathcal{H}$ contains the set of minimal elements of $\gen{\mathcal{H}}$, and so the above result is a slight generalization of Theorem~\ref{Uniform Kahn-Kalai Theorem} in the uniform case as $\mathcal{F} = \gen{\mathcal{F}}$ when $\mathcal{F}$ is an increasing family. For convenience, we define the transformations $T_{\ell}:[0,1]^X\rightarrow [0,1]^X$ for each $\ell\geq 1$ such that 
\begin{equation}\label{def:T}
T_{\ell}(\vec{q}) = \vec{p},\quad \text{where\ } p_x = 1-(1-q_x)^{4\floor{\log(2\ell)}+7} \text{ for all $x\in X$.}
\end{equation}

As in the uniform setting, we  write $V(\mathcal{H}, \vec{q})$ as an integer program, and then find its LP relaxation, denoted by $V_f(\mathcal{H},\vec{q})$, as follows.
\begin{alignat*}{7}
    V_f(\mathcal{H},\vec{q}) 
    = \min& &&\displaystyle\sum\limits_{S\subseteq X} g(S)\prod\limits_{x\in S} q_x
    \\
    \text{subject to}&\quad &&\sum\limits_{S\subseteq T} g(S)\geq 1 &&\hspace{3mm}&&\forall\hspace{1mm}  T\in\mathcal{H}\text{ and }
    \\
    &&& g(S)\geq 0 &&\hspace{3mm}&&\forall\hspace{1mm} S\subseteq X.
\end{alignat*}

\begin{remark*}
\normalfont
Let $\mathcal{H}\subseteq 2^X$ and let $\vec{q}\in [0,1]^X$. Then, $V(\mathcal{H},\vec{q})\geq V_f(\mathcal{H},\vec{q})$.
\end{remark*}

The dual of $V_f(\mathcal{H},\vec{q})$ is given below.
\begin{alignat*}{7}
    \max& &&\displaystyle\sum\limits_{T\in\mathcal{H}} \nu(T)
    \\
    \text{subject to}&\quad &&\nu(T)\geq 0 &&\hspace{3mm}&&\forall\hspace{1mm}  T\in\mathcal{H}
    \\ 
    & &&\sum\limits_{T\supseteq S} \nu(T)\leq\prod\limits_{s\in S} q_s &&\hspace{3mm}&&\forall\hspace{1mm} S\subseteq X.
\end{alignat*}
Motivated by the dual variables above we introduce the following definition of \vocab{spread} in the non-uniform setting.

\begin{definition*}
We say a probability measure $\nu$ supported on $\mathcal{H}$ is \vocab{$\vec{q}$-spread}  if
\begin{center}
    $\displaystyle\sum \limits_{T\supseteq S} \nu(T) \leq 2\prod\limits_{s\in S} q_s$ \quad for all $S\subseteq X$.  
\end{center}
We call such a measure a \vocab{$\vec{q}$-spread probability measure}.
\end{definition*}

The purpose of spread is the same as in the uniform case; Lemma~\ref{lemma: Analogue of Observation 7.4} is the non-uniform analogue of Observation~\ref{Observation 7.5 in Fractional Paper}.

\begin{lemma}
\label{lemma: Analogue of Observation 7.4}
Let $\mathcal{H}\subseteq 2^X$ and $\vec{q}\in [0,1]^X$. If $\mathcal{H}$ supports a $\vec{q}$-spread probability measure, then $V(\mathcal{H}, \vec{q})\geq 1/2$.
\end{lemma}

Again, for our applications, we will specifically focus on subgraph containment problems, and so we adopt the same $\mathcal{G}_H$ defined in Section~\ref{section: Uniform Spread Intro}, where $H\in\mathcal{K}_n^r$. In this setting, we again use $X = \binom{[n]}{r}$. Furthermore, let $\mathcal{F}_H = \gen{\mathcal{G}_H}$, the family of graphs which contain a copy of $H$ as a subgraph. Unlike the uniform case, the uniform measure on $\mathcal{G}_H$ often does not yield a good upper bound on $\vec{p}^*$, and so we are forced to search among a broader set of probability measures on $\mathcal{G}_H$ in order to obtain a good Kahn-Kalai type of upper bound. Recall $T_{\ell}$ from~(\ref{def:T}). 

\begin{theorem}
\label{Analogue of Observation 7.5}
Let $H\subseteq\mathcal{K}_n^r$ and $X = \binom{[n]}{r}$ and let $\mathcal{F}_H$  be the set of elements in $2^X$ that contains a copy of $H$ as a subset. Let $\pi$ be a probability measure on $S_n$ and let $\sigma\in S_n$ be sampled by $\pi$. Suppose that $\vec{q}\in [0,1]^X$ satisfies
\[
    \prob{\sigma(S)\subseteq H}\leq 2\displaystyle\prod\limits_{s\in S} q_s
\]
for all $S\subseteq\mathcal{K}_n^r$ isomorphic to a subhypergraph of $H$. If $\vec{p} = T_{\abs{E(H)}}(\vec{q})$, then $\mu_{\vec{p}}\left(\mathcal{F}_H\right)\geq 1/2$.
\end{theorem}

Next, we extend Theorem~\ref{Analogue of Observation 7.5} to a more general setting that forces members in $\mathcal{F}$ to all contain a set of fixed elements. 
For example, consider $X=\binom{[n]}{2}$ and $F=\{uv\}$. Let $\mathcal{F}_{\texttt{Ham}}$ denote the set of graphs on $[n]$ that contains a Hamilton cycle, and let 
\[
\mathcal{F}_{\texttt{Ham},F}=\{G\in \mathcal{F}_{\texttt{Ham}}: F\subseteq G\}.
\]
In other words, $\mathcal{F}_{\texttt{Ham},F}$ is the set of graphs that contains a Hamilton cycle $C$ such that $F\subseteq C$. With this setting, we are not interested in the probability that $\G(n,p)$ contains a member of $\mathcal{F}_{\texttt{Ham},F}$, since this probability is bounded by $p=\pr(uv\in \G(n,p))$, and there are no nontrivial phase transitions. What interests us is the probability that $\G(n,p)\cup F$ contains a member of $\mathcal{F}_{\texttt{Ham},F}$, which is exactly the probability that $\G(n,p)$ contains a Hamilton path that joins two specific vertices $u$ and $v$. There are many other applications where one can carefully choose $F$ and use them to patch different pieces of subgraphs, such as disjoint paths, to form larger structures such as Hamilton cycles. We explore this further in a subsequent paper. In this work, we generalize the setting so that we can estimate probabilities such as $\pr(\G(n,p)\cup F\in  \mathcal{F}_{\texttt{Ham},F})$.

Given an $r$-graph $H\subseteq\mathcal{K}_n^r$, and $F\subseteq H$, let
\[
\mathcal{F}_{H,F}=\{G\in \gen{\mathcal{G}_H}: F\subseteq G\}.
\]
In other words,
$\mathcal{F}_{H,F}$ denotes the set of graphs $G$ such that $G$ contains a subgraph $H'\cong H$ with $F\subseteq H'$. Let
\[
    S_n(F)=\{\sigma\in S_n: \sigma(u)=u\ \text{for all $u$ that is incident to some edge in $F$}\}
\]
and let $G\setminus H$  denote the hypergraph obtained by removing the edges of $H$ from $G$.

\begin{theorem}
\label{Hamilton Spread Permutation Theorem}
Let $H\subseteq\mathcal{K}_n^r$, $F\subseteq E(H)$, and let $X = \binom{[n]}{r}$. Let $\pi$ be a probability measure on $S_n(F)$ and let $\sigma\in S_n(F)$ be sampled by $\pi$. Suppose that $\vec{q}\in [0,1]^X$  satisfies
\[
    \prob{\sigma(R)\subseteq H}\leq 2\displaystyle\prod\limits_{r\in R\setminus F} q_r
\]
for all $F\subseteq R\subseteq X$ isomorphic to a subhypergraph of $H$. If $\vec{p} = T_{\abs{E(H)}}(\vec{q})$, then
\[
\pr(X_{\vec{p}}\cup F\in \mathcal{F}_{H,F})\ge 1/2.
\]
\end{theorem}

\begin{remark}
\label{remark: Hamilton Spread Is Generalization}
\normalfont
Theorem~\ref{Analogue of Observation 7.5} is the special case of Theorem~\ref{Hamilton Spread Permutation Theorem} with $F = \emptyset$.
\end{remark}

As mentioned earlier, we will explore the application of Theorem~\ref{Hamilton Spread Permutation Theorem} in a subsequent paper, where its strength becomes essential in the study of more challenging subgraph containment problems beyond perfect matchings.

The rest of the paper contains proofs of our main results. It is structured as follows: section~\ref{section: Proofs of Non-Uniform Spread Results} covers the proofs of Theorem~\ref{Analogue of Observation 7.5}, Theorem~\ref{Hamilton Spread Permutation Theorem}, Theorem~\ref{theorem: Non-Uniform Friedgut} and Lemma~\ref{lemma: Analogue of Observation 7.4}, while section~\ref{section: SBM proof} contains the proof of Theorem~\ref{theorem: General Stochastic Perfect Matching Improved}. Section~\ref{section: Proofs for Chung-Lu} contains the proofs for results related to the Chung-Lu model: section~\ref{section: Proof for Bi-Valued Chung-Lu} contains the proof of Theorem~\ref{theorem: Phase Transition for PM in Bi-Valued Families} while section~\ref{section: Proof of k-valued Chung-Lu} covers the proof of Example~\ref{example: 0-statement Chung-Lu k-valued}. Finally, section~\ref{section: Proofs for G(n,d)} contains the proofs for the results related to $\mathcal{G}(n,\vec{d})$: section~\ref{section: proof of G(n,d) 1-statement} covers the proof of Theorem~\ref{theorem: G(n,d) 1-statement}, section~\ref{section: Proof of G(n,d) Ideal 0-statement Counterexample} covers the proof of Example~\ref{example: G(n,d) Ideal 0-statement Counterexample}, section~\ref{section: proof of G(n,d) 0-statement} covers the proof of Theorem~\ref{theorem: G(n,d) bi-valued 0-statement}, and section~\ref{section: Proof of G(n,d) concrete threshold} covers the proof of Example~\ref{example: G(n,d) Ideal 0-statement Example}.

\section{Towards Non-Uniform Kahn-Kalai}
\label{section: Proofs of Non-Uniform Spread Results}

\subsection{Proofs of Theorem~\ref{Analogue of Observation 7.5} and Theorem~\ref{Hamilton Spread Permutation Theorem}}

We only present the proof of Theorem~\ref{Hamilton Spread Permutation Theorem} since Theorem~\ref{Analogue of Observation 7.5} follows directly by remark~\ref{remark: Hamilton Spread Is Generalization}. We first prove the following result, which is similar to the notion of $q$-spread probability measures in ~\cite{mossel2022secondkahnkalaiconjecture}.

\begin{theorem}
\label{Spreadness for Family of Graphs Theorem}
Let $X$ be a finite set. Let $\mathcal{T}\subseteq 2^X$ and let $\rho$ be an arbitrary probability distribution on $\mathcal{T}$. Then, $\rho$ is a $\vec{q}$-spread probability measure on $\mathcal{T}$ if and only if
\[
    \prob{S\subseteq\vec{H}}\leq\displaystyle 2\prod\limits_{s\in S} q_s
\]
for all $S\subseteq X$, where $\vec{H}\in \mathcal{T}$ is drawn according to distribution $\rho$.
\end{theorem}

\begin{proof}
First, recall that $\rho$ is a $\vec{q}$-spread probability measure on $\mathcal{T}$ if and only if
\begin{center}
    $\displaystyle\sum\limits_{\substack{T\supseteq S:\\T\in\mathcal{T}}} \rho(T)\leq 2\displaystyle\prod\limits_{s\in S} q_s$\quad for all $S\subseteq X$.
\end{center}
By definition, for all $S\subseteq X$, with $\vec{H}\sim\rho$,
\begin{center}
    $\prob{S\subseteq\vec{H}} = \displaystyle\sum\limits_{S\subseteq T\subseteq X} \rho(T) = \displaystyle\sum\limits_{T\supseteq S:\ T\in\mathcal{T}} \rho(T)$
\end{center}
where the last equality above holds by the fact that $\rho$ is supported on $\mathcal{T}$. Hence, $\rho$ is a $\vec{q}$-spread probability measure on $\mathcal{T}$ if and only if 
\begin{center}
    $\prob{S\subseteq\vec{H}}\leq 2\displaystyle\prod\limits_{s\in S} q_s$ \quad for all $S\subseteq X$,
\end{center}
 as desired.
\end{proof}

We now prove Theorem~\ref{Hamilton Spread Permutation Theorem}.

\begin{proof}
Let $\vec{q}'\in [0,1]^X$ such that $\vec{q}'_e = \vec{q}_e$ for all $e\in Y$, where $Y = X\setminus F$, and $\vec{q}'_e = 1$ for all $e\in F$. Let $\vec{p}' = T_{\abs{E(H)}}(\vec{q}')$. Note that $\vec{p}'_e = \vec{p}_e$ for all $e\in Y$ and $\vec{p}'_e = 1$ for all $e\in F$. It follows that $\prob{X_{\vec{p}}\cup F\in\mathcal{F}_{H,F}} = \mu_{\vec{p}'}(\mathcal{F}_{H,F})$ and so it suffices to prove that $\mu_{\vec{p}'}(\mathcal{F}_{H,F})\geq 1/2$.

Let $\mathcal{G}_{H,F} = \{\sigma(H): \sigma\in S_n(F)\}$.
Obviously, $\gen{\mathcal{G}_{H,F}}\subseteq \mathcal{F}_{H,F}$.
Define probability $\rho$ on $\mathcal{G}_{H,F}$ by
\[
    \rho(T)=\pr_{\sigma\sim \pi}(\sigma(T\setminus F)=H\setminus F),\quad \text{for every $T\in \mathcal{G}_{H,F}$.}
\]
Let $\vec{H}\sim\rho$. We show that
\begin{equation}
    \label{forced edge desired inequality}
    \prob{S\subseteq\vec{H}}\leq 2\prod\limits_{s\in S} q'_s \quad \text{for all $S\subseteq X$.}
\end{equation}
First, observe that we may assume that $F\subseteq S$ as otherwise the left hand side of~\eqref{forced edge desired inequality} is zero by the assumption that every $\sigma$ in the support of $\pi$ fixes all the vertices incident to some edge in $F$.
For any $F\subseteq S\subseteq X$,
\[
    \prob{S\subseteq\vec{H}} = \displaystyle\sum\limits_{\substack{T\in\mathcal{G}_{H,F},\\T\supseteq S}} \rho(T) = \displaystyle\sum\limits_{\substack{T\in\mathcal{G}_{H,F},\\T\supseteq S}} {\mathbb P}_{\sigma\sim \pi}(\sigma(T\setminus F)=H\setminus F) =\pr_{\sigma\sim\pi}(\sigma(S\setminus F)\subseteq H\setminus F),
\]
where the last equality above holds since $S\subseteq T$ if and only if $\sigma(S\setminus F)\subseteq\sigma(T\setminus F)$ by the definition of $S_n(F)$.  Furthermore, note that if $\sigma(S\setminus F)\subseteq H\setminus F$, then $\sigma(S)\subseteq H$ since $\sigma(F) = F$. It follows that
\[
    \prob{S\subseteq\vec{H}}\leq\prob{\sigma(S)\subseteq H}\leq 2\prod\limits_{s\in S\setminus F} q'_s = 2\prod\limits_{s\in S} q'_s
\]
by the theorem assumptions. Hence,~\eqref{forced edge desired inequality} holds for all $S\subseteq X$. 

By Theorem~\ref{Spreadness for Family of Graphs Theorem},  $\rho$ is a $\vec{q}'$-spread probability measure on $\mathcal{G}_{H,F}$. By Lemma~\ref{lemma: Analogue of Observation 7.4}, $V(\mathcal{G}_{H,F},\vec{q}')\geq 1/2$. Note that $\mathcal{G}_{H,F}$ is $\abs{E(H)}$-bounded because every element of $\mathcal{G}_{H,F}$ has $\abs{E(H)}$ edges. By Theorem~\ref{Restatement of P-V Theorem 1}, it follows that $1/2\leq \mu_{\vec{p}'}(\gen{\mathcal{G}_{H,F}})\leq\mu_{\vec{p}'}(\mathcal{F}_{H,F})$.
\end{proof}

\subsection{Proof of Theorem~\ref{theorem: Non-Uniform Friedgut}}
\label{section: Non-Uniform Friedgut Proof}

The proof of this theorem closely follows~\cite{Bollobs1987ThresholdF}.
Let $X_1, X_2,\ldots, X_k$ be independent copies of $X_{\vec{p}}$. Note that $X_1\cup X_2\cup\cdots\cup X_k$ is distributed as $X_{\vec{r}}$ where $\vec{r}\in [0,1]^X$ such that $r_x = 1-\left(1-p_x\right)^k$ for all $x\in X$. Observe further that $r_x\leq kp_x$ for all $x\in X$ and so
\begin{center}
    $\prob{X_{k\vec{p}}\notin\mathcal{F}}\leq \prob{X_{\vec{r}}\notin\mathcal{F}} = \left[\prob{X_{\vec{p}}\notin\mathcal{F}}\right]^k$.
\end{center}
Taking $\vec{p} = \vec{p}^*$  and $k = \alpha$ (recall that $\alpha=\omega(1)$ by the theorem assumption) above yields
\[
    \prob{X_{\alpha\vec{p}^*}\notin\mathcal{F}}\leq \left[ \prob{X_{\vec{p}^*}\notin\mathcal{F}}\right]^\alpha = \left[1-\mu_{\vec{p}^*}\left(\mathcal{F}\right)\right]^ \alpha < (1-\varepsilon)^\alpha = o(1),
\]
and so $\mu_{\alpha\vec{p}^*}(\mathcal{F})\rightarrow 1$. Similarly, we have
\[
    \varepsilon < \prob{X_{\vec{p}^*}\notin\mathcal{F}}\leq\left[\prob{X_{\vec{p}^*/\alpha}\notin\mathcal{F}}\right]^\alpha,
\]
which implies that
\[
    \prob{X_{\vec{p}^*/\alpha}\notin\mathcal{F}}\geq \varepsilon^{1/\alpha} = 1-o(1),
\]
and so $\mu_{\vec{p}^*/\alpha}(\mathcal{F})\rightarrow 0$.
\qed

\subsection{Proof of Lemma~\ref{lemma: Analogue of Observation 7.4}}
\label{section: Spread Lemma Proof}

Let $\nu'$ be a $\vec{q}$-spread probability measure supported on $\mathcal{F}$. Thus, 
\begin{center}
    $\displaystyle\sum\limits_{T\supseteq S} \nu'(T)\leq 2\prod\limits_{s\in S} q_s$ \quad \text{for all $S\subseteq X$.}
\end{center}
 Letting $\nu := \frac{1}{2}\nu'$, note that $\nu$ is a feasible solution of the dual of $V_f(\mathcal{F}, \vec{q})$, since $\nu(T)\geq 0$ for all $T\in\mathcal{F}$ by the definition of $\nu$ and the fact that $\nu'$ is a probability measure, and
\begin{center}
$\displaystyle\sum\limits_{T\supseteq S} \nu(T)\leq\prod\limits_{s\in S} q_s$, \quad for all $S\subseteq X$.
\end{center}
 Since $\nu'$ is a probability measure supported on $\mathcal{F}$, it follows that
\begin{center}
    $\displaystyle\sum\limits_{T\in\mathcal{F}} \nu(T) = \frac{1}{2}\sum\limits_{T\in\mathcal{F}} \nu'(T) = \frac{1}{2}$.
\end{center}
This means that the optimal value of the dual of $V_f(\mathcal{F}, \vec{q})$ is at least $1/2$. By weak duality, $V_f(\mathcal{F}, \vec{q})\geq 1/2$. Hence, $V(\mathcal{F},\vec{q})\geq 1/2$, as desired.
\qed

\section{Proof of Theorem~\ref{theorem: General Stochastic Perfect Matching Improved}}
\label{section: SBM proof}

In our proof we use the following approximation of falling factorials. 
\begin{lemma}
\label{lemma: Stirling Bound}
For integers $n\geq x\geq 0$ where $n\geq 1$, $(n)_x\geq (n/e)^x$.
\end{lemma}

\begin{proof}
We may assume $x\geq 1$ since $(n)_x\geq (n/e)^x$ trivially holds if $x = 0$. By~\cite{stirlingRemark}, we have $e^{1/(12n+1)} < n!/(\sqrt{2\pi n} (n/e)^n) < e^{1/(12n)}$ for all $n\geq 1$ and so $n!\geq (n/e)^n$ for all $n\geq 1$. It follows that
\[
    (n)_{x} = \binom{n}{x}x!\geq\left(\frac{x}{e}\right)^{x}\prod\limits_{j = 1}^{x}\frac{n-(j-1)}{x-(j-1)}\geq\left(\frac{x}{e}\right)^{x}\left(\frac{n}{x}\right)^{x} = \left(\frac{n}{e}\right)^{x},
\]
using the fact that each term in the product is at least $n/x$.
\end{proof}

Choose $\alpha = \omega(1)$ such that $\mathcal{G}_{\texttt{PM}}^\alpha(\vec{U}, P)$ has a perfect matching $H$ and for each $1\leq y,z\leq k$, define the indicator variables $\mathcal{I}(y,z) = \mathbbm{1}\left\{P(y,z)\geq\alpha\log n/\max(n_y, n_z)\right\}$.
We say $\sigma\in S_n$ is a \vocab{block invariant permutation} if $\sigma(U_i) = U_i$ for all $1\leq i\leq k$. Let $\pi$ be a uniform distribution over all block invariant permutations in $S_n$. Let $\vec{q}\in (0,1)^{\binom{[n]}{2}}$ be defined as follows: for each $uv\in \binom{[n]}{2}$, let $q_{uv} = \frac{30}{\max(n_y,n_z)}\cdot\mathcal{I}(y,z)$ where $y,z\in\{1,2,\ldots,k\}$ such that $u\in U_y$ and $v\in U_z$. We show the following claim and then apply Theorem~\ref{Analogue of Observation 7.5} to get the desired result.
Recall that $\mathcal{K}_n^2$ denotes the set of edges in the complete graph on $n$ vertices.

\begin{claim}
\label{claim: PM Spread Inequality}
Let $\sigma$ be sampled according to the distribution $\pi$. Then
\begin{center}
    $\prob{\sigma(S)\subseteq H}\leq 2\displaystyle\prod\limits_{s\in S} q_s$
\end{center}
for all $S\subseteq\mathcal{K}_n^2$ isomorphic to a subhypergraph of $H$.
\end{claim}

\begin{proof}
Let $H_{ij}$ be the set of edges in $H$ with one end in $U_i$ and the other end in $U_j$, for all $1\le i\le j\le k$.  Let $S\subseteq\mathcal{K}_n^2$ be isomorphic to a subgraph of $H$, and so, $S$ is a set of vertex-disjoint edges. Similarly, let $S_{ij}$ be the set of edges in $S$ with one end in $U_i$ and the other hand in $U_j$, for all $1\le i\le j\le k$. Letting $s_{ij} = \abs{S_{ij}}$ and $h_{ij} = \abs{H_{ij}}$, we may assume that $s_{ij}\leq h_{ij}$ for all $1\leq i,j\leq k$ since $\prob{\sigma(S)\subseteq H} = 0$ otherwise. Let $f(S,H)$ denote the number of block-invariant permutations $\sigma$ for which $\sigma(S)\subseteq H$. Therefore,\begin{equation}\label{eqn:Permutation Ratio for Stirling}
    \prob{\sigma(S)\subseteq H} = \displaystyle\frac{f(S,H)}{\prod\limits_{i = 1}^{k} n_i!}.
\end{equation}
To estimate $f(S,H)$, first note that for each $1\leq i,j\leq k$, there are $(h_{i,j})_{s_{i,j}}$ choices for $\sigma$ such that
$\sigma(S_{i,j})\subseteq H_{i,j}$.
Each such choice specifies the image of $V(S)$ under $\sigma$, where $V(S)$ is the set of vertices incident with $S$.
Then, to complete a specification of $\sigma$,
there are $(n_i - \zeta(i))!$ ways to permute the unmapped vertices in each $U_i$, where $\zeta(i):= 2s_{i,i} + \sum\limits_{j\neq i} s_{i,j}$ for each $1\leq i\leq k$. It follows that
\begin{center}
    $\prob{\sigma(S)\subseteq H} = \displaystyle\frac{\prod\limits_{j\geq i} (h_{i,j})_{s_{i,j}}\prod\limits_{i = 1}^{k} (n_i - \zeta(i))!}{\prod\limits_{i = 1}^{k} n_i!} = \frac{\prod\limits_{j\geq i} (h_{i,j})_{s_{i,j}}}{\prod\limits_{i = 1}^{k} (n_i)_{\zeta(i)}}$.
\end{center}

Since $h_{i,j}\leq\min(n_i,n_j)$, we get that $\left(h_{i,j}\right)_{s_{i,j}}\leq\min(n_i,n_j)^{s_{i,j}}$ for all $1\leq i,j\leq k$. Using Lemma~\ref{lemma: Stirling Bound}, we have
\[
    \prob{\sigma(S)\subseteq H}\leq\left[\prod\limits_{j\geq i} \min(n_i,n_j)^{s_{i,j}}\right]\prod\limits_{i = 1}^{k}  \left(\frac{e}{n_i}\right)^{\zeta(i)} = \prod\limits_{j\geq i} \left[\min(n_i,n_j)^{s_{i,j}}\left(\frac{e}{n_i}\cdot\frac{e}{n_j}\right)^{s_{i,j}}\right].
\]
Therefore, we get
\begin{center}
    $\prob{\sigma(S)\subseteq H}\leq\displaystyle\prod\limits_{j\geq i} \left(\frac{e^2}{\max(n_i,n_j)}\right)^{s_{i,j}} < \prod\limits_{j\geq i} \prod\limits_{s\in S_{i,j}} \frac{30}{\max(n_i,n_j)}$
\end{center}
since $e^2 < 30$. Hence, the desired inequality holds.
\end{proof}

Now, we will use the result of Claim~\ref{claim: PM Spread Inequality} to finish the proof of Theorem~\ref{theorem: General Stochastic Perfect Matching Improved}. 
Define $\vec{p}\in [0,1]^X$ (recall $X = \binom{[n]}{2}$) as follows. For every $uv\in X$, let $p_{uv} = P(y,z)$ where $y,z\in\{1,\ldots,k\}$ are such that $u\in U_y$ and $v\in U_z$. 
Then, by the assumption of the theorem and the definitions of $\vec{q}$ and $T_{\ell}$ from~\ref{def:T},
\begin{center}
    $p_{uv} \geq \alpha\log n\cdot\displaystyle\frac{30}{\max(n_y, n_z)}\cdot\mathcal{I}(y,z) = \alpha q_{uv}\log n \geq \alpha T_{n/2}(\vec{q})_{uv}$,
\end{center}
where the last inequality above is tight up to a constant factor. Therefore $\vec{p}\geq \alpha T_{n/2}(\vec{q})$ which implies that $\mu_{\vec{p}}(\mathcal{F}_H)\rightarrow 1$ by Theorem~\ref{theorem: Non-Uniform Friedgut} and Theorem~\ref{Analogue of Observation 7.5} because $\abs{E(H)} = \frac{n}{2}$. Hence $\mathcal{B}(\vec{U}, P)$ a.a.s.\ has a perfect matching, as desired.\hfill \qedsymbol

\section{Perfect matchings in the Chung-Lu model}
\label{section: Proofs for Chung-Lu}

\subsection{Proof of Theorem~\ref{theorem: Phase Transition for PM in Bi-Valued Families}}
\label{section: Proof for Bi-Valued Chung-Lu}

\noindent Recall that $U_1 = \{v_x: d_x = n_1\}$ and let $U_2 = \{v_x: d_x = n_2\}$. 
Recall that the families $\mathcal{D}_1$ and $\mathcal{D}_2$ were defined as follows.
\begin{itemize}
    \item $\vec{d}\in\mathcal{D}_1$ if $n_1\geq n_2 \ge n^{\Theta(1)}$;
    \item $\vec{d}\in\mathcal{D}_2$ if $n_1\leq n^{\delta}$ for some fixed $\delta < 1$.
\end{itemize}

Our goal is to show that for $\vec{d}\in\mathcal{D}_1\cup\mathcal{D}_2$, if $\mathfrak{G}_{\texttt{PM}}^{\alpha}\left(\vec{U}_{\vec{d}}, P_{\vec{d}}\right)$ does not have a perfect matching for some $\alpha = o(1)$, then $\mathcal{G}_{\text{CL}}(n, \vec{d})$ a.a.s.\ does not have a perfect matching. For bi-valued degree sequences $\vec{d}$, we have the following straightforward observation. 
\begin{lemma}
\label{lemma: bi-valued PM equivalence}
$\mathfrak{G}_{\texttt{PM}}^{\alpha}\left(\vec{U}_{\vec{d}}, P_{\vec{d}}\right)$ has a perfect matching if and only if 
\begin{center}
    $\displaystyle\frac{d_x d_2}{\norm{\vec{d}}_1}\geq \frac{\alpha\log n}{n_x}$
\end{center}
where $x\in\{1,2\}$ is such that $n_x = \max\{n_1, n_2\}$, and $x$ is defined to be 1 if $n_1=n_2$.
\end{lemma}

\begin{proof}
First, suppose $d_x d_2/\norm{\vec{d}}_1\geq\alpha\log n/n_x$ and let $y = 3-x$ (i.e., $n_y = \min\{n_1, n_2\}$). Note that
\begin{center}
    $\displaystyle\frac{d_1 d_2}{\norm{\vec{d}}_1}\geq\frac{d_x d_2}{\norm{\vec{d}}_1}\geq\frac{\alpha\log n}{\max(n_1, n_2)}$
\end{center}
and 
\begin{center}
    $\displaystyle\frac{d_x^2}{\norm{\vec{d}}_1}\geq\frac{d_x d_2}{\norm{\vec{d}}_1}\geq\frac{\alpha\log n}{\max(n_x, n_x)}$. 
\end{center}
By the definition of $\mathfrak{G}_{\texttt{PM}}^{\alpha}\left(\vec{U}_{\vec{d}}, P_{\vec{d}}\right)$, every vertex of $U_x$ is adjacent to every other vertex in the graph (i.e. they each have $n-1$ neighbours). We can construct a perfect matching of $\mathfrak{G}_{\texttt{PM}}^{\alpha}\left(\vec{U}_{\vec{d}}, P_{\vec{d}}\right)$ by matching all vertices in $U_y$ with vertices in $U_x$, and then matching the remaining vertices in $U_x$ among themselves.

Now, suppose that $\mathfrak{G}_{\texttt{PM}}^{\alpha}\left(\vec{U}_{\vec{d}}, P_{\vec{d}}\right)$ has a perfect matching. If $x = 2$, then it follows that $d_2^2/\norm{\vec{d}}_1\geq\alpha\log n/n_2$, since otherwise, there will be no edges between any two vertices in $U_2$, and consequently $\mathfrak{G}_{\texttt{PM}}^{\alpha}\left(\vec{U}_{\vec{d}}, P_{\vec{d}}\right)$ cannot have a perfect matching because $\abs{U_2}>\abs{U_1}$, leading to a contradiction. If $x = 1$, then suppose for the sake of contradiction that $d_1d_2/\norm{\vec{d}}_1 < \alpha\log n/n_1$. This implies that $d_2^2/\norm{\vec{d}}_1 < \alpha\log n/n_2$.  Hence, all vertices in $U_2$ are isolated, contradicting with $\mathfrak{G}_{\texttt{PM}}^{\alpha}\left(\vec{U}_{\vec{d}}, P_{\vec{d}}\right)$ having a perfect matching. 
\end{proof}

By Lemma~\ref{lemma: bi-valued PM equivalence}, it suffices to show that if $\vec{d}\in\mathcal{D}_1\cup\mathcal{D}_2$ is such that $d_x d_2/\norm{\vec{d}}_1\ll\log n/n$, then a.a.s.\ $\mathcal{G}_{\text{CL}}(n,\vec{d})$ has no perfect matching. We prove the statement separately for $\vec{d}\in\mathcal{D}_1$, and $\vec{d}\in\mathcal{D}_2$. 

\subsubsection{Case 1: \pdfmath{\vec{d}\in\mathcal{D}_1}}
In this case $x=1$.
Let $p = d_1 d_2/\norm{\vec{d}}_1\ll\log n/n$ and let $q = d_2^2/\norm{\vec{d}}_1\leq p$. Let $X$ be the number of isolated vertices in $U_2$. We show that a.a.s.\ $X>0$ which implies that a.a.s.\ $\G_{\text{CL}}(n,\vec{d})$ has no perfect matching.
\[
\ex{X}= n_2 \left(1-q\right)^{n_2-1}\left(1-p\right)^{n_1} > n_2\left(1-p\right)^n = n_2\exp(-np+O(np^2))=\omega(1), 
\]
where the last equality holds since $p = o(\log n/n)$ and $n_2\ge n^{\Theta(1)}$. Moreover,
\[
\ex{ X(X-1)}= n_2\left(n_2 - 1\right) \left(1-q\right)^{n_2 - 1}\left(1-p\right)^{n_1} \left(1-q\right)^{n_2 - 2}\left(1-p\right)^{n_1}\sim (\ex X)^2.
\]
It follows that a.a.s.\ $X>0$ by Chebyshev's inequality.

\subsubsection{Case 2: \pdfmath{\vec{d}\in\mathcal{D}_2}}
In this case $x=2$ and $d_2^2/\norm{\vec{d}}_1\ll\log n/n$. Let $X$ be the number of vertices in $U_2$ that are not adjacent to any other vertices in $U_2$. We show that a.a.s.\ $X>n_1$, which then immediately implies that a.a.s.\ $\mathcal{G}_{\text{CL}}(n,\vec{d})$ has no perfect matchings.
\[
\ex X = n_2\displaystyle\left(1 - \frac{d_2^2}{\norm{\vec{d}}_1}\right)^{n_2-1} = \exp\left(\log n -n_2 \frac{d_2^2}{\norm{\vec{d}}_1} +O\left(1+\frac{d_2^4n}{\norm{\vec{d}}_1^2}\right)\right) =n^{1-o(1)},
\]
where the last equality above holds since $d_2^2/\norm{\vec{d}}_1\ll\log n/n$ and $n_2\sim n$.
Moreover,
\[
\ex X(X-1)=n_2\left(n_2 - 1\right)\left(1-\frac{d_2^2}{\norm{\vec{d}}_1}\right)^{2n_2-3} \sim (\ex X)^2.
\]
By Chebyshev's inequality, a.a.s.\ $X=n^{1-o(1)}>n_1$, as desired.  \qed

\subsection{Proof of Example~\ref{example: 0-statement Chung-Lu k-valued}}
\label{section: Proof of k-valued Chung-Lu}

Choose $\alpha = o(1)$ such that $\mathfrak{G}_{\texttt{PM}}^{\alpha}\left(\vec{U}_{\vec{d}}, P_{\vec{d}}\right)$ has no perfect matchings. Note that $ \max(n_1, n_x)=n_1$ for all $1\leq x\leq k$. Furthermore, note that $d_1 d_k/\norm{\vec{d}}_1 < \alpha\log n/n_1$, since otherwise, $d_1 d_x/\norm{\vec{d}}_1\geq\alpha\log n/\max(n_1, n_x)$ for all $1\leq x\leq k$ and consequently we can form a perfect matching by matching each vertex in $[n]\setminus U_1$ to a vertex in $U_1$ and then matching up the remaining vertices in $U_1$. This also implies that $d_x d_k/\norm{\vec{d}}_1 < \alpha\log n/n_1$ for all $1\leq x\leq k$. Let $p = \alpha\log n/n_1 = o(\log n/n)$ since $n_1 = \Theta(n)$.

Let $X$ be the number of isolated vertices in $U_k$. We show that a.a.s.\ $X > 0$ which implies that a.a.s.\ $\mathcal{G}_{\text{CL}}(n,\vec{d})$ has no perfect matching. It is easy to see that
\[
\ex X = n_k\prod\limits_{x = 1}^{k} \left(1 - \frac{d_x d_k}{\norm{\vec{d}}_1}\right)^{n_x-\mathbbm{1}\{x = k\}} > n_k\left(1-p\right)^n = n_k\exp(-np + O(np^2)) = \omega(1),
\]
where the last equality holds since $p = o(\log n/n)$ and $n_k \ge n^{\delta}$. As before, it is easy to see that $\ex X(X-1)\sim (\ex X)^2$ and so it follows that a.a.s.\ $X > 0$ by Chebyshev's inequality.
\qed

\section{Perfect matchings in \pdfmath{\G(n,\vec{d})}}
\label{section: Proofs for G(n,d)}

We apply Theorem 1.7 from~\cite{Gao_2024}, which is restated below, that allows us to embed the Chung-Lu model within $\mathcal{G}(n,\vec{d})$.

\begin{theorem}[Gao-Ohapkin~\cite{Gao_2024}]
\label{theorem: Jane's G(n,d) embedding}
Given $\vec{d} = (d_1,\ldots,d_n)$ with $d_1\geq\cdots\geq d_n$, let $P(\vec{d})$ be the symmetric $n\times n$ matrix defined by $P_{ij} = P_{ji} = d_i d_j/\norm{\vec{d}}_1$, for every $1\leq i < j\leq n$, and $P_{ii} = 0$ for every $1\leq i\leq n$. Assume that $\vec{d} = \vec{d}(n)$ is a degree sequence that satisfies $d_1^2 = o(\norm{\vec{d}}_1)$ and $d_n\gg\log n$. Then, there exists $\varepsilon = o(1)$ and a coupling $(G_L, G)$ such that marginally, $G_L\sim\mathcal{G}(n, (1-\varepsilon)P(\vec{d}))$ and $G\sim\mathcal{G}(n,\vec{d})$, and jointly, $\prob{G_L\subseteq G} = 1-o(1)$. 
\end{theorem}

For convenience, we say a.a.s.\ $\mathcal{G}(n, (1-\varepsilon) P(\vec{d}))\subseteq\mathcal{G}(n,\vec{d})$, meaning that there exists a coupling as described in Theorem~\ref{theorem: Jane's G(n,d) embedding} (and use analogous notation for other random graph models).

\subsection{Proof of Theorem~\ref{theorem: G(n,d) 1-statement}}
\label{section: proof of G(n,d) 1-statement}

Let $\varepsilon=o(1)$ be chosen so that Theorem~\ref{theorem: Jane's G(n,d) embedding} holds.
Let $\alpha = \omega(1)$ be such that $\mathfrak{G}_{\texttt{PM}}^{\alpha}(\vec{U}_{\vec{d}}, P_{\vec{d}})$ has a perfect matching. Letting $\alpha' = 2\alpha/(1-\varepsilon)$ and $\vec{d}' = ((1-\varepsilon)/2)  \vec{d}$, note that $\mathfrak{G}_{\texttt{PM}}^{\alpha'}(\vec{U}_{\vec{d}'}, P_{\vec{d}'})=\mathfrak{G}_{\texttt{PM}}^{\alpha}(\vec{U}_{\vec{d}}, P_{\vec{d}})$ (to see this, note that $d_i d_j/\norm{\vec{d}}_1\geq\alpha\log n/\max(n_x, n_y)$ if and only if $d'_i d'_j/\norm{\vec{d}'}_1\geq\alpha'\log n/\max(n_x, n_y)$, where $v_i\in U_x$ and $v_j\in U_y$). It follows immediately that $\mathfrak{G}_{\texttt{PM}}^{\alpha'}(\vec{U}_{\vec{d}'}, P_{\vec{d}'})$ has a perfect matching, and since $\alpha' = \omega(1)$, it follows that $\mathcal{G}_{\text{CL}}(n,\vec{d}')$ has a perfect matching a.a.s.\ by Theorem~\ref{theorem: General Stochastic Perfect Matching Improved}. Finally, by Theorem~\ref{theorem: Jane's G(n,d) embedding}, there exists a coupling such that a.a.s.\ $\mathcal{G}_{\text{CL}}(n,\vec{d}')\subseteq\mathcal{G}(n,\vec{d})$, which further implies that a.a.s.\ $\mathcal{G}(n,\vec{d})$ has a perfect matching, as desired. \qed

\subsection{Proof of Example~\ref{example: G(n,d) Ideal 0-statement Counterexample}}
\label{section: Proof of G(n,d) Ideal 0-statement Counterexample}

Let $K>0$ be sufficiently large.
By Lemma~\ref{lemma: bi-valued PM equivalence},  $\mathfrak{G}_{\texttt{PM}}^{\alpha}(\vec{U}_{\vec{d}}, P_{\vec{d}})$ has no perfect matching for $\alpha = Kn^{-1/4}$, since $d_2^2/\norm{\vec{d}}_1 =O( n^{1/4}\log n/n^{3/2}) \le \alpha\log n/n$. Let $\varepsilon = o(1)$ from Theorem~\ref{theorem: Jane's G(n,d) embedding} and $\vec{d}' := \frac{1-\varepsilon}{2}\vec{d}$. To complete the proof that a.a.s.\ $\G(n,\vec{d})$ has a perfect matching, it suffices to show that a.a.s.\ $\mathcal{G}_{\texttt{CL}}(n, \vec{d}')$ has a perfect matching. We first expose the subgraph induced by $U_2$. It is easy to see that a.a.s.\ $U_2$ induces at least an edge. Take an edge $uv$ in $U_2$.  Then we consider the subgraph of $\mathcal{G}_{\texttt{CL}}(n, \vec{d}')$ joining $U_1$ and $U_2\setminus \{u,v\}$ and it suffices to show that a.a.s.\ this subgraph has a perfect matching. Note that this subgraph has the distribution $\mathcal{G}(n_1,n_1,p)$ for some $p\ge 2\log n/n$ since
\begin{center}
    $\displaystyle\frac{d'_1 d'_2}{\norm{\vec{d}'}_1} = \frac{(1-\varepsilon) d_1 d_2}{2\norm{\vec{d}}_1} =\Omega\left( \frac{n^{5/8}\sqrt{\log n}}{n^{3/2}} \right)\gg\frac{\log n}{n}$.
\end{center}
By~\cite{bipartitePM}, $\mathcal{G}(n_1,n_1,p)$ a.a.s.\ has a perfect matching. It follows now that $\mathcal{G}_{\text{CL}}(n,\vec{d}')$ a.a.s.\ has a perfect matching.
\qed

\subsection{Proof of Theorem~\ref{theorem: G(n,d) bi-valued 0-statement}}
\label{section: proof of G(n,d) 0-statement}

The 1-statement of Theorem~\ref{theorem: G(n,d) bi-valued 0-statement} directly follows from Theorem~\ref{theorem: G(n,d) 1-statement} by Remark~\ref{remark: G(n,d) remark}. 
Now we show the 0-statement. We will make use of a special case of Theorem 4.6 from~\cite{McKay_asymptotics}, which we restate below.

\begin{theorem}[McKay~\cite{McKay_asymptotics}]
\label{theorem: Mckay's Deg Seq Enumeration Result}
For a sequence $\vec{d} = (d_1,\ldots,d_n)$ of positive integers $d_1\geq d_2\geq\cdots\geq d_n$, let $\hat{\Delta} = 2 + d_1(\frac{3}{2} d_1 +1)$ and $\lambda = \frac{1}{2\norm{\vec{d}}_1}\sum\limits_{i = 1}^{n} d_i(d_i-1)$.  If $\hat{\Delta} < \frac{1}{3}\norm{\vec{d}}_1$, then the number of graphs with degree sequence $\vec{d}$ is
\begin{center}
    $\displaystyle\frac{\norm{\vec{d}}_1!}{(\norm{\vec{d}}_1/2)!\hspace{1mm} 2^{\norm{\vec{d}}_1/2}\prod\limits_{i=1}^{n} d_i!}\exp\left(-\lambda - \lambda^2 + O\left(\hat{\Delta}^2/\norm{\vec{d}}_1\right)\right)$.
\end{center}
\end{theorem}

Let $\vec{d}$ be the bi-valued degree sequence in Theorem~\ref{theorem: G(n,d) bi-valued 0-statement}.
Note that the assumption  $\hat{\Delta} < \frac{1}{3}\norm{\vec{d}}_1$ of Theorem~\ref{theorem: Mckay's Deg Seq Enumeration Result} is satisfied since $\hat{\Delta}^2 = \Theta(d_1^4) = o(\norm{\vec{d}}_1)$ by condition~\ref{cond: stronger d_1 bound}. Let $X$ be the number of vertices $v\in U_2$ that have no neighbours in $U_2$. We show that a.a.s.\ $X > n_1$, which then immediately implies that a.a.s.\ $\mathcal{G}(n,\vec{d})$ has no perfect matchings. Let $D(\vec{d})$ denote the number of graphs with degree sequence $\vec{d}$ and by Theorem~\ref{theorem: Mckay's Deg Seq Enumeration Result}, we have
\[
D(\vec{d}) = \displaystyle\frac{\norm{\vec{d}}_1!}{(\norm{\vec{d}}_1/2)!\hspace{1mm} 2^{\norm{\vec{d}}_1/2} \prod\limits_{i = 1}^{n} d_i!} \exp\left(-\lambda - \lambda^2 + O\left(d_1^4/\norm{\vec{d}}_1\right)\right)
\]
where $\lambda  =\lambda(\vec{d}):= \sum\limits_{i = 1}^{n} d_i (d_i-1)/(2\norm{\vec{d}}_1)$.

For each $v\in U_2$ and $Y\subseteq U_1$ such that $|Y|=d_2$, we estimate $D^{v,Y}(\vec{d})$, the number of graphs with degree sequence $\vec{d}$ such that $N(v)=Y$. Let $\vec{d}^{v,Y}$ denote the degree sequence on $[n]\setminus \{v\}$ such that $d^{v,Y}_k = d_k$ for all $k\notin \{v\}\cup Y$ and $d^{v,Y}_k = d_k - 1$ for all $k\in Y$. Then, the set of graphs with degree sequence $\vec{d}$ such that $N(v)=Y$ is bijective to the set of graphs on $[n]\setminus \{v\}$ with degree sequence $\vec{d}^{v,Y}$. Hence, $D^{v,Y}(\vec{d})=D(\vec{d}^{v,Y})$.
Hence,
\[
    \ex{X} = \sum\frac{D(\vec{d}^{v,Y})}{D(\vec{d})},
    \]
where the sum is over all $(v,Y)$ such that 
$v\in U_2$ and $Y\subseteq U_1$ with $|Y|=d_2$.
Obviously $D^{v,Y}(\vec{d})$ are independent of $(v,Y)$ due to symmetry. Hence,
\[    
    \ex{X} = n_2\binom{n_1}{d_2} \frac{D(\vec{d}^{v,Y})}{D(\vec{d})}.
\]
By Theorem~\ref{theorem: Mckay's Deg Seq Enumeration Result}, with  $\varphi = (\lambda(\vec{d}) - \lambda(\vec{d}^{v,Y})) + (\lambda(\vec{d})^2 - (\lambda(\vec{d}^{v,Y}))^2)$ (noting that $\varphi$ is independent of $(v,Y)$ and similarly for $\gamma$ and $\theta$ below) and
\[
\gamma= \left(\frac{\norm{\vec{d}^{v,Y}}_1!}{(\norm{\vec{d}^{v,Y}}_1/2)!\hspace{1mm} 2^{\norm{\vec{d}^{v,Y}}_1/2} }\right)\left(\frac{\norm{\vec{d}}_1!}{(\norm{\vec{d}}_1/2)!\hspace{1mm} 2^{\norm{\vec{d}}_1/2}}\right)^{-1} \quad\text{ and }\quad \theta = \frac{\prod\limits_{i = 1}^{n} d_i!}{\prod\limits_{\substack{1\leq i\leq n,\\ i\neq v}} d^{v,Y}_i!},
\]    
we obtain that    
    \[\ex{X} =  n_2\binom{n_1}{d_2}\theta \gamma\exp\left(\varphi + O\left(d_1^4/\norm{\vec{d}^{v,Y}}_1\right)\right)\sim n_2\binom{n_1}{d_2}\theta\gamma e^{\varphi}.
\]
where the last asymptotic follows from condition~\ref{cond: stronger d_1 bound} and $\norm{\vec{d}^{v,Y}}_1 = \norm{\vec{d}}_1 - 2d_2$. Note that $\theta = d_1^{d_2} d_2!$ and
\[
\gamma = 2^{d_2}\cdot\frac{(\norm{\vec{d}}_1/2)_{d_2}}{(\norm{\vec{d}}_1)_{2d_2}} = 2^{d_2}\cdot\frac{(\norm{\vec{d}}_1/2)^{d_2}}{\norm{\vec{d}}_1^{2d_2}}\exp\left(O\left(\frac{d_2^2}{\norm{\vec{d}}_1}\right)\right)\sim\frac{1}{\norm{\vec{d}}_1^{d_2}}
\]
since $d_2^2 = o(\norm{\vec{d}}_1)$ by condition~\ref{cond: stronger d_1 bound}.
Now, we have
\begin{eqnarray*}
\ex{X} &\sim &n_2\binom{n_1}{d_2} d_1^{d_2}d_2!\frac{e^{\varphi}}{\norm{\vec{d}}_1^{d_2}} \sim n_2\left(\frac{n_1 d_1}{\norm{\vec{d}}_1}\right)^{d_2} = n_2\left(1 - \frac{n_2 d_2}{\norm{\vec{d}}_1}\right)^{d_2}\\
&\sim& n_2\exp\left(-\frac{n_2 d_2^2}{\norm{\vec{d}}_1}\right) = n^{1-o(1)},
\end{eqnarray*}
where the second asymptotic follows from $\binom{n_1}{d_2}\sim  n_1^{d_2}/d_2!$ by condition~\ref{cond: d_2 and n_1 bound} and the following claim, while the last equality follows from $d_2^2/\norm{\vec{d}}_1\ll\log n/n$ and $n_2\sim n$.
\begin{claim}
\label{claim: first moment exp vanishes}
$\varphi = o(1)$.
\end{claim}
To estimate $\ex X(X-1)$, for each $u,v\in U_2$ and $X,Y\subseteq U_1$, let $\vec{d}^{u,v,X,Y}$ denote the degree sequence on $[n]\setminus \{u,v\}$ where $d^{u,v,X,Y}_k = d_k$ for all $k\notin \{u,v\}\cup X\cup  Y$, $d^{u,v,X,Y}_k = d_k-1$ for all $k\in (X\setminus Y)\cup (Y\setminus X)$, and $d^{u,v,X,Y}_k = d_k-2$ for all $k\in X\cap Y$. As before, let $D^{u,v,X,Y}(\vec{d})$ denote the number of graphs  with degree sequences $\vec{d}$ such that $N(u)=X$ and $N(v)=Y$, and it follows that $D^{u,v,X,Y}(\vec{d})=D(\vec{d}^{u,v,X,Y})$. Hence,
\[
\ex{X(X-1)} = \sum \frac{D(\vec{d}^{u,v,X,Y})}{D(\vec{d})}
\]
where the sum is over all $(u,v,X,Y)$ such that 
$u,v\in U_2$ and $X,Y\subseteq U_1$ with $\abs{X} = \abs{Y}=d_2$.
We split the sum according to the size of $X\cap Y$ and obtain that
\[
\ex{X(X-1)} = \sum\limits_{w = 0}^{d_2}\sum \frac{D(\vec{d}^{u,v,X,Y})}{D(\vec{d})},
\]
where the second sum is over all $(u,v,X,Y)$ such that 
$u,v\in U_2$, $X,Y\subseteq U_1$, $\abs{X} = \abs{Y}=d_2$ and $\abs{X\cap Y} = w$. Obviously, given $w=|X\cap Y|$, $D(\vec{d}^{^{u,v,X,Y}})$ is independent of $(u,v,X,Y)$ due to symmetry. Hence,
\[
\ex{X(X-1)} = n_2(n_2-1)\binom{n_1}{d_2}\sum\limits_{w = 0}^{d_2} \binom{d_2}{w}\binom{n_1-d_2}{d_1-w}\frac{D(\vec{d}^{u,v,X,Y})}{D(\vec{d})}.
\]
Let $\varphi_w = (\lambda(\vec{d}) - (\lambda(\vec{d}^{u,v,X,Y})) + (\lambda(\vec{d})^2 - (\lambda(\vec{d}^{u,v,X,Y}))^2)$, noting that $\varphi_w$ is independent of $(u,v,X,Y)$, and let
\[
\gamma= \left(\frac{\norm{\vec{d}^{u,v,X,Y}}_1!}{(\norm{\vec{d}^{u,v,X,Y}}_1/2)!\hspace{1mm} 2^{\norm{\vec{d}^{u,v,X,Y}}_1/2} }\right)\left(\frac{\norm{\vec{d}}_1!}{(\norm{\vec{d}}_1/2)!\hspace{1mm} 2^{\norm{\vec{d}}_1/2}}\right)^{-1} \quad\text{ and }\quad \theta_w = \frac{\prod\limits_{i = 1}^{n} d_i!}{\prod\limits_{\substack{1\leq i\leq n,\\ i\neq u,v}} (d^{u,v,X,Y})_i!},
\]
by noting that $\theta_w$ is independent of $(u,v,X,Y)$ while $\gamma$ is independent of $(u,v,X,Y)$ \textit{and} $w$, since $\norm{\vec{d}^{u,v,X,Y}}_1 = \norm{\vec{d}}_1 - 4d_2$. Applying Theorem~\ref{theorem: Mckay's Deg Seq Enumeration Result}, we have
\[
\ex{X(X-1)} = \gamma\ n_2(n_2-1)\binom{n_1}{d_2}\sum\limits_{w = 0}^{d_2} \binom{d_2}{w}\binom{n_1-d_2}{d_1-w} \theta_w\exp\left(\varphi_w + O\left(\frac{d_1^4}{\norm{\vec{d}^{u,v,X,Y}}_1}\right)\right)
\]
since $\norm{\vec{d}^{u,v,X,Y}}_1 = \norm{\vec{d}}_1-4d_2\sim\norm{\vec{d}}_1$ by condition~\ref{cond: stronger d_1 bound}, and this condition also implies that
\[
\ex{X(X-1)}\sim\gamma n_2^2\binom{n_1}{d_2}\sum\limits_{w = 0}^{d_2} \binom{d_2}{w}\binom{n_1-d_2}{d_1-w} \theta_w e^{\varphi_w}.
\]
Note that $\theta_w = d_1^{2d_2-w} (d_1-1)^w (d_2!)^2$ and
\[
\gamma = 2^{2d_2}\cdot \frac{(\norm{\vec{d}}_1/2)_{2d_2}}{(\norm{\vec{d}}_1)_{4d_2}} = 2^{2d_2}\cdot\frac{(\norm{\vec{d}}_1/2)^{2d_2}}{\norm{\vec{d}}_1^{4d_2}}\exp\left(O\left(\frac{d_2^2}{\norm{\vec{d}}_1}\right)\right)\sim\frac{1}{\norm{\vec{d}}_1^{2d_2}}
\]
since $d_2^2 = o(\norm{\vec{d}}_1)$ by condition~\ref{cond: stronger d_1 bound}. Letting $\varphi':= \max\limits_{0\leq w\leq d_2} \varphi_w$, we have
\[
\ex{X(X-1)}\sim n_2^2\left(\frac{n_1 d_1}{\norm{\vec{d}}_1}\right)^{2d_2}\sum\limits_{w = 0}^{d_2} \left(\frac{d_1-1}{d_1}\right)^w\binom{d_2}{w}\binom{n_1-d_2}{d_1-w} e^{\varphi_w}\leq n_2^2\left(\frac{n_1 d_1}{\norm{\vec{d}}_1}\right)^{2d_2} e^{\varphi'}
\]
using Vandermonde's identity. Then it follows that $\ex{X(X-1)}\leq (1+o(1))(\ex{X})^2$ by the following claim.

\begin{claim}
\label{claim: second moment exp vanishes}
$\varphi' = o(1)$.
\end{claim}

\noindent Finally, by Chebyshev's inequality, a.a.s.\ $X = n^{1-o(1)} > n_1$, as desired.
\qed

\subsection{Proof of Claims~\ref{claim: first moment exp vanishes} and~\ref{claim: second moment exp vanishes}}

Before proving the two claims, we first prove the following intermediate lemma. Recall that $\vec{d}$ satisfies the conditions in Theorem~\ref{theorem: G(n,d) bi-valued 0-statement}.

\begin{lemma}
\label{lemma: G(n,d) 0-statement Intermediate Lemma}
Let $A = \sum\limits_{i = 1}^{n} d_i(d_i-1)$, $T= O(d_1^2)$, and $B = A - T$. Furthermore, let $A' = A/(2\norm{\vec{d}}_1)$ and $B' = B/(2\norm{\vec{d}}_1 - k d_2)$ for some $k = \Theta(1)$. If $Z = (A' - B')(1 + A' + B')$, then $Z = o(1)$.
\end{lemma}

\begin{proof}
Note that
\[
A' - B' = \frac{A}{2\norm{\vec{d}}_1} - \frac{A-T}{2\norm{\vec{d}}_1 - k d_2} = \frac{2T\norm{\vec{d}}_1 - kA d_2}{2\norm{\vec{d}}_1(2\norm{\vec{d}}_1 - k d_2)} = \frac{T}{2\norm{\vec{d}}_1-k d_2} - A'\cdot\frac{k d_2}{\norm{\vec{d}}_1-k d_2},
\]
and so it follows that
\[
A' + B' = A' + \left(1 +\frac{k d_2}{\norm{\vec{d}}_1-k d_2}\right)A' - \frac{T}{2\norm{\vec{d}}_1-k d_2} = \frac{2\norm{\vec{d}}_1 - k d_2}{\norm{\vec{d}}_1 - k d_2}\cdot A' - \frac{T}{2\norm{\vec{d}}_1-k d_2}.
\]
Note that $A' = \Theta((n_1 d_1^2 + n_2 d_2^2)/(n_1 d_1 + n_2 d_2))= \Omega(1)$. Furthermore, since $n_1 d_1 \leq n^{\delta}\sqrt{d_2} = o(n_2 d_2)$ by conditions~\ref{cond: stronger d_1 bound} and ~\ref{cond: n_1 small}, we have
\begin{equation}
\label{eqn: A' bound}
    A' = O\left(\frac{n_1 d_1^2 +n_2 d_2^2}{n_2 d_2}\right) = O\left(\frac{n_1 d_1^2}{nd_2}+ d_2\right).
\end{equation}
Since $T = O(d_1^2) = o(\norm{\vec{d}}_1)$ by~\ref{cond: stronger d_1 bound}, $k d_2 = o(\norm{\vec{d}}_1)$, and $A'=\Omega(1)$, we have
\begin{align*}
    Z &= \left(\frac{T}{2\norm{\vec{d}}_1-k d_2} - A'\cdot\frac{k d_2}{\norm{\vec{d}}_1-k d_2}\right)\left(1 + \frac{2\norm{\vec{d}}_1 - k d_2}{\norm{\vec{d}}_1 - k d_2}\cdot A' - \frac{T}{2\norm{\vec{d}}_1-k d_2}\right)
    \\
    &= \left(O\left(\frac{T}{\norm{\vec{d}}_1}\right) - O\left(\frac{d_2 A'}{\norm{\vec{d}}_1}\right)\right)\left(O\left(A'\right) - O\left(\frac{T}{\norm{\vec{d}}_1}\right)\right)
    \\
    &= O\left(\frac{T}{\norm{\vec{d}}_1} A'+\frac{T^2}{\norm{\vec{d}}_1^2}+\frac{d_2}{\norm{\vec{d}}_1} (A')^2+\frac{T d_2}{\norm{\vec{d}}_1^2} A'\right)
    \\
    &=O\left(\frac{T}{\norm{\vec{d}}_1} A'+\frac{d_2}{\norm{\vec{d}}_1} (A')^2\right),
\end{align*}
where the last equality follows from the fact that $T/\norm{\vec{d}}_1 = o(1) = o(A')$. By~\eqref{eqn: A' bound},
we have
\[
\frac{d_2 (A')^2}{\norm{\vec{d}}_1} = O\left(\frac{n_1^2 d_1^4}{n^2 d_2\norm{\vec{d}}_1} +\frac{d_2^3}{\norm{\vec{d}}_1}\right) = O\left(\frac{n_1^2 d_1^4}{n^2 d_2\norm{\vec{d}}_1} +\frac{d_1^4}{d_2\norm{\vec{d}}_1}\right) = O\left(\frac{d_1^4}{d_2\norm{\vec{d}}_1}\right) = o(1),
\]
using the fact that $n_1^2/n^2\leq n^{2\delta-2} = o(1)$ by condition~\ref{cond: n_1 small} and $d_1^4 = o(\norm{\vec{d}}_1)$ by condition~\ref{cond: stronger d_1 bound}. Furthermore, we have
\[
\frac{T}{\norm{\vec{d}}_1} A' = O\left(\frac{n_1 d_1^4}{n d_2\norm{\vec{d}}_1} + \frac{d_1^2 d_2}{\norm{\vec{d}}_1}\right) = O\left(\frac{d_1^4}{\norm{\vec{d}}_1}\right) = o(1),
\]
using the fact that $n_1/(nd_2)\leq n^{\delta-1} = o(1)$ by condition~\ref{cond: n_1 small} and $d_1^4 = o(\norm{\vec{d}}_1)$  by condition~\ref{cond: stronger d_1 bound}. Hence, $Z = o(1)$, as desired.
\end{proof}

\subsubsection{Proof of Claim~\ref{claim: first moment exp vanishes}}

Let $A = \sum\limits_{i = 1}^{n} d_i(d_i-1)$ and $B = \sum\limits_{\substack{1\leq i\leq n,\\ i\neq v}} d^{v,Y}_i (d^{v,Y}_i-1)$ and note that
\begin{center}
    $B = \displaystyle\sum\limits_{i\notin\{v\}\cup Y} d_i(d_i-1) + \sum\limits_{i\in Y} (d_i-1)(d_i-2) = A - d_2(d_2-1) - 2d_2 (d_1-1)$
\end{center}
since $d_i = d_1$ for all $i\in Y$ and $\abs{Y} = d_2$. Furthermore, note that $B = A - T$ where $T = d_2 (d_2 + 2d_1 - 3) = O(d_1^2)$, and so by Lemma~\ref{lemma: G(n,d) 0-statement Intermediate Lemma} with $k = 2$, it follows that $\varphi = o(1)$, as desired.\qed

\subsubsection{Proof of Claim~\ref{claim: second moment exp vanishes}}

It suffices to show that $\varphi_w = o(1)$ for all $0\leq w\leq d_2$. Let $w\in\{0,1,\ldots,d_2\}$ be arbitrary and let $(u,v,X,Y)$ be subject to the constraints in Section~\ref{section: proof of G(n,d) 0-statement} such that $\abs{X\cap Y} = w$.
Let 
\begin{align*}
 A &= \sum\limits_{i = 1}^{n} d_i(d_i-1)\\  B &= \sum\limits_{\substack{1\leq i\leq n,\\ i\neq u,v}} d^{u,v,X,Y}_i (d^{u,v,X,Y}_i-1) \\
 S_1 &= \{u,v\}\cup X\cup Y\\
 S_2&= (X\setminus Y)\cup (Y\setminus X)\\
 S_3 &= X\cap Y.
\end{align*}
Note that
\begin{align*}
    B &= \sum\limits_{i\notin S_1} d_i(d_i-1) + \sum\limits_{i\in S_2} (d_i-1)(d_i-2) +\sum\limits_{i\in S_3} (d_i-2)(d_i-3)
    \\
    &= A - 2d_2(d_2-1) - 4(d_2-w) (d_1-1) - w(4d_1-6)
    \\
    &= X - 2d_2(d_2-1) - 4d_2(d_1-1) + 2w
\end{align*}
since $d_i = d_1$ for all $i\in X\cup Y$ and $\abs{X\setminus Y} = \abs{Y\setminus X} = d_2-w$. Furthermore, note that $B = A - T$ where $T = 2d_2 (d_2 + d_1 - 3) + 2w = O(d_1^2)$, and so by Lemma~\ref{lemma: G(n,d) 0-statement Intermediate Lemma} with $k = 8$, it follows that $\varphi_w = o(1)$, as desired.\qed

\subsection{Proof of Example~\ref{example: G(n,d) Ideal 0-statement Example}}
\label{section: Proof of G(n,d) concrete threshold}

Note that $\vec{d}$ satisfies conditions~\ref{cond: d_2 and n_1 bound}, \ref{cond: stronger d_1 bound} and~\ref{cond: n_1 small} (with $\delta = 15/16$).
By Lemma~\ref{lemma: bi-valued PM equivalence}, $\mathfrak{G}_{\texttt{PM}}^{\alpha}(\vec{U}_{\vec{d}}, P_{\vec{d}})$ has a perfect matching if and only if $d_2^2/\norm{\vec{d}}_1\geq\alpha\log n/n_2$ since $n_2 > n_1$. Note that
\begin{center}
    $\displaystyle\frac{d_2^2}{\norm{\vec{d}}_1} = \frac{d_2^2}{n^{17/16} + (n - n^{15/16}) d_2}\sim \frac{d_2^2}{n^{17/16} + nd_2}$
\end{center}
If $d_2\gg n^{1/16}$, then $d_2^2/\norm{\vec{d}}_1\sim d_2/n\gg\log n/n$. Otherwise,
\begin{center}
    $\displaystyle\frac{d_2^2}{\norm{\vec{d}}_1}\sim\frac{d_2^2}{n^{17/16}} = \frac{\log n}{n}\cdot\frac{d_2^2}{n^{1/16}\log n}$,
\end{center}
and so $d_2^2/\norm{\vec{d}}_1\gg\log n/n$ if $d_2\gg n^{1/32}\sqrt{\log n}$, and $d_2^2/\norm{\vec{d}}_1\ll\log n/n$ if $d_2\ll n^{1/32}\sqrt{\log n}$. 

This implies that $\mathfrak{G}_{\texttt{PM}}^{\alpha}(\vec{U}_{\vec{d}}, P_{\vec{d}})$ has a perfect matching for some $\alpha = \omega(1)$ if $d_2\gg n^{1/32}\sqrt{\log n}$. By Theorem~\ref{theorem: G(n,d) 1-statement}, it follows that a.a.s.\ $\mathcal{G}(n,\vec{d})$ has a perfect matching. On the other hand, if $d_2\ll n^{1/32}\sqrt{\log n}$, we have
\begin{center}
    $\displaystyle\frac{n_2^2 d_2^3}{\norm{\vec{d}}^2_1}\sim\frac{d_2^3}{(n^{1/16}+d_2)^2}\leq \frac{d_2^3}{n^{1/8}} = o(1)$.
\end{center}
Hence, $\vec{d}$ satisfies all conditions of Theorem~\ref{theorem: G(n,d) bi-valued 0-statement} (A1)--(A5). It follows then that a.a.s.\ $\mathcal{G}(n,\vec{d})$ has no perfect matching. \qed

\bibliographystyle{plain}
\bibliography{references}

\end{document}